\documentclass[a4paper]{amsart}
\usepackage{wrapfig}
\usepackage[dvips]{graphicx}
\usepackage{amsmath,amsthm,amssymb,amscd}
\theoremstyle{definition}
\newtheorem{thm}{Theorem}[section]
\newtheorem{Def}[thm]{Definition}
\newtheorem{pro}[thm]{Proposition}
\newtheorem{cor}[thm]{Corollary}
\newtheorem{lem}[thm]{Lemma}
\newtheorem{ex}[thm]{Example}
\newtheorem{rem}[thm]{Remark}
\newtheorem{que}[thm]{Question}
\theoremstyle{definition}

\begin{document}

\title[Finite group actions on certain stably projectionless C$^*$-algebras]{Finite group actions on certain stably projectionless C$^*$-algebras with the Rohlin property}
\author{Norio Nawata}
\address{Department of Mathematics and Informatics, 
Graduate school of Science,
Chiba University,1-33 Yayoi-cho, Inage, Chiba, 263-8522, Japan}
\email{nawata@math.s.chiba-u.ac.jp}
\thanks{The author is a Research Fellow of the Japan Society for the Promotion of Science.}
\begin{abstract} 
We introduce the Rohlin property and the approximate representability for finite group actions on stably projectionless C$^*$-algebras 
and study their basic properties. We give some examples of finite group actions on the Razak-Jacelon algebra $\mathcal{W}_2$ 
and show some classification results of these actions. This study is based on the work of Izumi, Robert's classification theorem and 
Kirchberg's central sequence C$^*$-algebras. 
\end{abstract}
\maketitle

\section{Introduction} 
A C$^*$-algebra $A$ is said to be \textit{stably projectionless} if $A\otimes\mathbb{K}$ has no non-zero projections where $\mathbb{K}$ is the C$^*$-algebra of compact operators on an infinite-dimensional 
separable Hilbert space. Note that every stably projectionless C$^*$-algebra does not have a unit. In this paper we study finite group actions on stably projectionless C$^*$-algebras. 

In the theory of operator algebras, the study of group actions on operator algebras is one of the most fundamental subjects and has a long history. 
We refer the reader to \cite{I} and references given there for this subject. 
Connes considered the Rohlin property for automorphisms of von Neumann algebras and classified 
automorphisms of the AFD factor of type II$_1$ in \cite{C} and \cite{C2}. It is important to consider the Rohlin property for classifying group actions on operator algebras. 
Using a suitable formulation of the Rohlin property of automorphisms of unital C$^*$-algebras, Kishimoto classified a large class of automorphisms on 
UHF algebras and certain A$\mathbb{T}$ algebras in \cite{Kis1} and \cite{Kis2}. 
Nakamura showed that every aperiodic automorphism of a Kirchberg algebra has the Rohlin property and classified them in \cite{Nak}. 
Izumi, Katsura and Matui showed the classification results of large classes of $\mathbb{Z}^N$-actions (see \cite{IM}, \cite{KM}, \cite{M1} and \cite{M2}). 
Izumi defined the Rohlin property for finite group actions on unital C$^*$-algebras (see also \cite{Kis0} and \cite{HJ}) and classified a large class of finite group actions on 
unital C$^*$-algebras in \cite{I1} and \cite{I2}. 

In the above formulations of the Rohlin property, a partition of unity consisting of projections (in the central sequence C$^*$-algebra) is used. 
Hence it is not clear how to define the Rohlin property for group actions on stably projectionless C$^*$-algebras. 
Although Evans and Kishimoto defined the Rohlin property for automorphisms of non-unital C$^*$-algebras and classified trace scaling automorphisms of certain AF algebras 
in \cite{EK}, their definition requires that C$^*$-algebras have non-zero projections. 
Sato defined the weak Rohlin property for automorphisms of unital projectionless C$^*$-algebras and 
classified a large class of automorphisms of the Jiang-Su algebra $\mathcal{Z}$ in \cite{Sa}. 
Matui and Sato defined the weak Rohlin property for countable amenable group actions and classified large classes of actions of $\mathbb{Z}^2$ and 
the Klein bottle group on $\mathcal{Z}$ in \cite{MS1} and \cite{MS2}. We do not need a partition of unity consisting of projections in the definition of the weak Rohlin 
property. But the weak Rohlin property is too weak from the viewpoint of the classification of finite group actions. 
Moreover their arguments are based on analyses of the actions on UHF algebras with the usual Rohlin property. 
Therefore we need to consider a suitable formulation of the Rohlin property for actions on stably projectionless C$^*$-algebras. 

We shall define the Rohlin property for finite group actions on $\sigma$-unital C$^*$-algebras by using Kirchberg's central sequence C$^*$-algebras 
defined in \cite{Kir2}.  
Kirchberg's central sequence algebra $F(A)$ of a C$^*$-algebra $A$ is defined as the quotient C$^*$-algebra of the central sequence C$^*$-algebra $A_{\infty}$ 
by the annihilator of $A$. Kirchberg's central sequence C$^*$-algebras are useful for proving $\mathcal{Z}$-stability results for non-unital C$^*$-algebras (see \cite{Na2} and \cite{Tik}). 
One of the purposes of this paper is showing that Kirchberg's central sequence C$^*$-algebras are also useful for the classification of group actions on 
certain stably projectionless C$^*$-algebras. 

Let $\mathcal{W}_{2}$ be the Razak-Jacelon algebra studied in \cite{J}, which is a certain simple nuclear stably projectionless C$^*$-algebra having trivial $K$-groups
and a unique tracial state and no unbounded traces. The Razak-Jacelon algebra $\mathcal{W}_2$ is regarded as a stably finite analogue of the Cuntz algebra $\mathcal{O}_2$.  
Moreover Robert showed that  $\mathcal{W}_{2}\otimes \mathbb{K}$ is isomorphic to $\mathcal{O}_2\rtimes_{\alpha}\mathbb{R}$ for some one parameter automorphism group $\alpha$ in \cite{Rob} 
(see also \cite{KK1}, \cite{KK2} and \cite{Dean}). 
Hence it is natural and interesting to consider whether $\mathcal{W}_2$ has the similar properties of $\mathcal{O}_2$. 

In this paper we study mainly finite group actions on $\mathcal{W}_2$. This is based on Izumi's study of finite group actions on $\mathcal{O}_2$ in \cite{I1}. 
This paper is organized as follows: In Section \ref{sec:Pre}, we collect notations and some results. 
In Section \ref{sec:Rohlin}, we define the Rohlin property for finite group actions on $\sigma$-unital C$^*$-algebras and study their basic properties. 
In particular, we show that any two actions of a finite group on $\mathcal{W}_2$ with the Rohlin property are conjugate (Corollary \ref{cor:Razak-Rohlin}). 
We also show that there exist strong $K$-theoretical constraints for the Rohlin property as in the unital case. 
In Section \ref{sec:app-rep}, we introduce the approximate representability for finite group actions on $\sigma$-unital C$^*$-algebras, which is the "dual" notion of the Rohlin property. 
We show a classification result of certain approximately representable actions (Corollary \ref{cor:app-classification}). 
Using this classification result, we give interesting examples of finite group actions on $\mathcal{W}_2$ with the Rohlin property (Corollary \ref{thm:rohlin-example}). 
In Section \ref{sec:sym}, we classify certain "locally representable" actions of $\mathbb{Z}_2$ on $\mathcal{W}_2$ up to conjugacy (Theorem \ref{thm:Main-sym}) and up to 
cocycle conjugacy (Theorem \ref{thm:cocycle-conjugacy}). 
We also construct a locally representable action $\alpha$ of $\mathbb{Z}_2$ on $\mathcal{W}_2$ such that 
$K_0(\mathcal{W}_2\rtimes_{\alpha} \mathbb{Z}_2)=\mathbb{Z}[\frac{1}{2}]$ and $K_1(\mathcal{W}_2\rtimes_{\alpha} \mathbb{Z}_2)=0$.

\section{Preliminaries}\label{sec:Pre}
In this section we shall collect notations and some results. We refer the reader to \cite{Bla} and \cite{Ped2} for basic facts of operator algebras.

\subsection{Notations}
We say that a C$^*$-algebra $A$ is $\sigma$-{\it unital} if $A$ has a countable approximate unit. 
In particular, if $A$ is $\sigma$-unital, then there exists a positive element $s\in A$ such that $\{s^{\frac{1}{n}}\}_{n\in\mathbb{N}}$ is an approximate unit. 
Such a positive element $s$ is called {\it strict positive} in $A$. 
If $A$ is separable, then $A$ is $\sigma$-unital. 
We denote by  $\tilde{A}$ the unitization algebra of $A$. 
A \textit{multiplier algebra}, denoted by $M(A)$, of $A$ is the largest unital C$^*$-algebra that contains $A$ as an essential ideal. 
It is unique up to isomorphism over $A$. If $\alpha$ is an automorphism of $A$, then $\alpha$ extends uniquely to an automorphism $\bar{\alpha}$ of $M(A)$. 
A homomorphism $\phi$ from $A$ to $B$ is said to be \textit{nondegenerate} if $\phi (A)B$ is dense in $B$. 
A nondegenerate homomorphism from $A$ to $B$ can be uniquely extended to a homomorphism $\bar{\phi}$ from $M(A)$ to $M(B)$. 
Note that if $A\subset B$ is an nondegenerate inclusion, then every approximate unit for $A$ is an approximate unit for $B$. 

For a unitary element $u$ in $M(A)$, define an automorphism $\mathrm{Ad}(u)$ of $A$ by $\mathrm{Ad}(u) (a)=uau^*$ for $a\in A$, and such an 
automorphism is called an \textit{inner automorphism}. Let $\mathrm{Aut}(A)$ denote the automorphism group of $A$, which is equipped with 
the topology of pointwise norm convergence. An automorphism $\alpha$ is said to be \textit{approximately inner} if $\alpha $ is in the closure of the inner automorphism group. 
We say that two automorphisms $\alpha$ and $\beta$ are \textit{approximately unitarily equivalent} if $\alpha\circ \beta^{-1}$ is approximately inner. 

For a locally compact group $G$, an \textit{action} $\alpha$ of $G$ on $A$ is a continuous homomorphism from $G$ to $\mathrm{Aut}(A)$. 
We say that $\alpha$ is \textit{outer} if $\alpha_g$ is not an inner automorphism for any $g\in G\setminus \{\iota \}$ where $\iota$ is the identity of $G$. 
An $\alpha$-\textit{cocycle} is a continuous map $u$ from $G$ to the unitary group $U(M(A))$ of $M(A)$ such that $u(gh)=u(g)\alpha_{g}(u(h))$ for any $g,h\in G$. 
If $u$ is an $\alpha$-cocycle, we define the action $\alpha^{u}$ of $G$ on $A$ by $\alpha^{u}_g (a)=\mathrm{Ad}(u(g))\circ \alpha_g(a)$ for any $a\in A$ and $g\in G$. 
For two $G$-actions $\alpha$ on $A$ and $\beta$ on $B$, we say that $\alpha$ and $\beta$ are \textit{conjugate},  written $(A,\alpha)\cong (B,\beta)$, 
if there exists an isomorphism $\theta$ from $A$ onto $B$ 
such that $\theta \circ \alpha_g=\beta_{g}\circ \theta$ for any $g\in G$. 
They are said to be \textit{cocycle conjugate} if there exists an $\alpha$-cocycle $u$ such that $\alpha^{u}$ is conjugate to $\beta$. 

Let $T(A)$ be the set of densely defined lower semicontinuous traces on $A$ and $T_1(A)$ the set of tracial states on $A$, and put 
$T_0(A):= \{\tau \in T(A)\ | \ \|\tau \| \leq 1  \}$. If $A$ is unital, then $T_1(A)$ is a Choquet simplex (see \cite[Theorem 3.1.18]{Sak}). 
There exists a natural one to one correspondence between $T_0(A)$ and $T_1 (\tilde{A})$, and hence $T_0(A)$ is a Choquet simplex. 
Every tracial state $\tau$ on $A$ extends uniquely to a tracial state $\bar{\tau}$ on $M(A)$ (see, for example, \cite[Proposition 2.9]{Tik}). 
For $\tau\in T_1(A)$, consider the Gelfand-Naimark-Segal (GNS) construction $(\pi_{\tau}, H_{\tau}, \xi)$. 
Then $\tau$ extends uniquely to a normal tracial state $\tilde{\tau}$ on $\pi_{\tau} (A)^{''}$. 
Note that if $\tau$ is an extremal tracial state, then $\pi_{\tau} (A)^{''}$ is a finite factor. 
Let $\alpha$ be an automorphism of $A$ such that $\tau \circ \alpha =\tau$. 
Then $\alpha$ extends uniquely to an automorphism $\tilde{\alpha}$ of $\pi_{\tau} (A)^{''}$. 
Moreover if $\alpha$ is an action of $G$ on $A$ such that  $\tau \circ \alpha_g =\tau$ for any $g\in G$, then 
$\alpha$ extends uniquely to a von Neumann algebraic action $\tilde{\alpha}$ on $\pi_{\tau}(A)^{''}$. 
We say that an action $\alpha$ of $G$ on a C$^*$-algebra $A$ with a unique tracial state $\tau$ 
is \textit{strongly outer} if $\tilde{\alpha}_g$ is not inner in $\pi_{\tau}(A)^{''}$ for any $g\in G\setminus \{\iota\}$. 

We denote by $K(H)$, $\mathbb{K}$ and $M_{n^\infty}$ for $n\in\mathbb{N}$ the C$^*$-algebra of compact operators on a Hilbert space $H$, the C$^*$-algebra of compact operators on an infinite-dimensional 
separable Hilbert space and the uniformly hyperfinite (UHF) algebra of type $n^{\infty}$, respectively. 
Let $\mathbb{Z}_n$ denote the cyclic group of order $n$. For a $\mathbb{Z}_n$-action $\alpha$, $\alpha_1$ is denoted by $\alpha$ for simplicity. 
If $G$ is a finite group, we denote by $| G|$ the order of $G$. 

\subsection{Crossed products and fixed point algebras}\label{sec:crossed}
For an action $\alpha$ of $G$ on $A$, we denote by $A\rtimes_{\alpha} G$ and $A^{\alpha}$ the reduced crossed product C$^*$-algebra and the fixed point algebra, respectively. 
If $G$ is a compact group, then $A^{\alpha}$ is isomorphic to a conner algebra of $A\rtimes_{\alpha} G$ \cite{Ros}. 
Indeed, define a projection $e_{\alpha}$ in $M(A\rtimes_{\alpha}G)$ by 
$$
e_{\alpha}:=\int_{G}\lambda_{g} dg
$$
where $\lambda_g$ is the implementing unitary of $\alpha_g$ in $M(A\rtimes_{\alpha}G)$ and $dg$ stands for the normalized Haar measure. 
Then $A^{\alpha}$ is isomorphic to $e_{\alpha}(A\rtimes_{\alpha} G)e_{\alpha}$. Note that we have $M(A^{\alpha})=M(A)^{\bar{\alpha}}$. 

When $G$ is a locally compact abelian group, let $\hat{\alpha}$ denote the dual action of $\alpha$ on $A\rtimes_{\alpha} G$. 
The Takesaki-Takai duality theorem \cite{Ta} shows that there exists an isomorphism $\Phi$ from 
$A\rtimes_{\alpha}G\rtimes_{\hat{\alpha}}\hat{G}$ onto $A\otimes K(\ell^2(G))$ such that $\Phi\circ\hat{\hat{\alpha}}=(\alpha\otimes \mathrm{Ad}(\rho ) ) \circ\Phi$ where 
$\rho$ is the right regular representation of $G$. 
If $G$ is a finite abelian group, 
then we see that $\bar{\Phi}(e_{\alpha})=1_{M(A)}\otimes e$ where $e\in K(\ell^2(G))$ is the projection onto 
the constant functions and may regard $\Phi (a)$ as a diagonal matrix 
$$
 \left(\begin{array}{cccc}
            a    &     &  &  \\ 
                 &  \alpha_{g_1}(a)   &  &  \\
                 &                    & \ddots  &  \\
                 &                    &         & \alpha_{g_{|G|-1}}(a)
 \end{array} \right) 
$$
where $a\in A$ and $G=\{\iota,g_1,...,g_{|G|-1}\}$. 

Assume that $\tau$ is a tracial state on $A$ such that $\tau \circ \alpha_g =\tau$ for any $g\in G$. 
Then we can construct a von Neumann algebraic crossed product $W^*(\pi_{\tau}(A)^{''}, \tilde{\alpha}, G)$. 
If $A$ is simple, then we may consider 
$$
A\rtimes_{\alpha}G \subset M(A\rtimes_{\alpha}G) \subset W^*(\pi_{\tau}(A)^{''}, \tilde{\alpha}, G)
$$
and $A\rtimes_{\alpha}G$ is weakly dense in $W^*(\pi_{\tau}(A)^{''}, \tilde{\alpha}, G)$ because $\pi_{\tau}$ is a faithful representation of $A$. 
The following proposition is well-known (see, for example, \cite{Bed}). 

\begin{pro}\label{pro:strongly outer}
Let $A$ be a simple C$^*$-algebra with a unique tracial state $\tau$, and let $\alpha$ be a strongly outer action of a discrete group $G$ on $A$. 
Then $A\rtimes_{\alpha}G$ has a unique tracial state $E\circ \tau$ where $E$ is the canonical conditional expectation from $A\rtimes_{\alpha}G$ 
onto $A$. 
\end{pro}

Hiroki Matui told us the following lemmas. 

\begin{lem}\label{lem:compact}
Let $A\subset B$ be a nondegenerate inclusion of separable C$^*$-algebras. 
If $T_1(A)$ is compact, then $T_1(B)$ is compact. 
\end{lem}
\begin{proof}
On the contrary, suppose that $T_1(B)$ were not compact. 
Then there exist a sequence $\{\tau_n\}_{n\in\mathbb{N}}$ in $T_1(B)$ and an element $\tau$ in $T_0(B)$ such that 
$\tau _n$ converges to $\tau$ (in the weak-$^*$ topology) and $\| \tau \| <1$. 
Since $T_1(A)$ is compact, the restriction of $\tau$ on $A$ is a tracial state on $A$. 
This is a contradiction because $A\subset B$ is a nondegenerate inclusion. 
\end{proof}

\begin{lem}\label{lem:two extremal}
Let $B$ be a C$^*$-algebra, and let $\beta$ be an action of a finite group $G$ on $B$. 
Assume that $T_1(B)$ is compact and there exists a unique tracial state $\tau$ on $B$ such that 
$\tau \circ \beta_{g} = \tau$ for any $g\in G$. Then $B$ has at most $|G|$ extremal tracial states. 
\end{lem}
\begin{proof}
By the Krein-Milman theorem, $T_1(B)$ is the closed convex hull of its extreme points $\mathrm{ex}(T_1(B))$. 
Let $\tau_{1}$ be an extremal tracial state on $B$. Then $\tau_1 \circ \beta_g$ is an extremal tracial state for any $g\in G$. 
Since $\tau$ is the unique tracial state on $B$ such that $\tau \circ \beta_g = \tau$ for any $g\in G$, 
we have $\tau =\frac{1}{|G|}\sum_{g\in G}\tau_1 \circ \beta_g$. 
On the contrary, suppose that $\mathrm{ex}(T_1(B))$ contained at least $|G |+1$ points. 
Then there exists an extremal tracial $\tau_{2}$ such that $\tau_2\neq \tau_1\circ \beta_g$ for any $g\in G$, and 
we have $\tau =\frac{1}{|G|}\sum_{g\in G}\tau_2\circ \beta_g $. 
This is a contradiction because $T_0(B)$ is a Choquet simplex. 
\end{proof}

\begin{pro}
Let $A$ be a separable C$^*$-algebra with a unique tracial state $\tau$, and let $\alpha$ be an action of a finite abelian group on $A$. 
Then $A\rtimes_{\alpha}G$ has  at most $|G|$ extremal tracial states. 
\end{pro}
\begin{proof}
Lemma \ref{lem:compact} shows that $T_1(A\rtimes_{\alpha}G)$ is compact. 
Since $A\rtimes_{\alpha}G\rtimes_{\hat{\alpha}}\hat{G}$ is isomorphic to $A\otimes M_{|G|}(\mathbb{C})$, 
$A\rtimes_{\alpha}G\rtimes_{\hat{\alpha}}\hat{G}$ has a unique tracial state. 
Hence we see that the fixed point of $\hat{\alpha}$ in $T_1(A\rtimes_{\alpha}G)$ is only one point. 
Therefore $A\rtimes_{\alpha}G$ has at most $|G|$ extremal tracial states by Lemma \ref{lem:two extremal}. 
\end{proof}

\begin{pro}\label{pro:two extremal}
Let $A$ be a simple separable C$^*$-algebra with a unique tracial state $\tau$, and let $\alpha$ be an action of $\mathbb{Z}_2$ on $A$. 
Then $A\rtimes_{\alpha}\mathbb{Z}_2$ has exactly two extremal tracial states if and only if 
there exists a unitary element $U_{\alpha}$ in $\pi_{\tau}(A)^{''}$ such that $\tilde{\alpha} =\mathrm{Ad}(U_{\alpha})$ and $U_{\alpha}^2 =1$. 
\end{pro}
\begin{proof}
Assume that there exists a unitary element $U_{\alpha}$ in $\pi_{\tau}(A)^{''}$ such that $\tilde{\alpha} =\mathrm{Ad}(U_{\alpha})$ and $U_{\alpha}^2 =1$. 
Then $W^*(\pi_{\tau}(A)^{''}, \tilde{\alpha}, \mathbb{Z}_2)$ is isomorphic to $\pi_{\tau}(A)^{''}\oplus\pi_{\tau}(A)^{''}$. 
Hence we see that $W^*(\pi_{\tau}(A)^{''}, \tilde{\alpha}, \mathbb{Z}_2)$ has two extremal normal tracial states $\sigma_1$ and $\sigma_2$ 
such that $\sigma_1 (m_0+m_1\lambda ) =\tilde{\tau}(m_0) +\tilde{\tau}(m_1U_{\alpha})$ and $\sigma_2 (m_0+m_1\lambda ) =\tilde{\tau}(m_0) -\tilde{\tau}(m_1U_{\alpha})$ 
where $\lambda$ is the implementing unitary of $\alpha$ and $m_0,m_1\in \pi_{\tau}(A)^{''}$. 
Since $A\rtimes_{\alpha}\mathbb{Z}_2$ is weakly dense in $W^*(\pi_{\tau}(A)^{''}, \tilde{\alpha}, \mathbb{Z}_2)$ and $\pi_{\tau}(A)^{''}$ is a factor, 
we see that the restrictions of $\sigma_1$ and $\sigma_2$ on $A\rtimes_{\alpha}\mathbb{Z}_2$ are distinct extremal tracial states on $A\rtimes_{\alpha}\mathbb{Z}_2$. 
By the proposition above, $A\rtimes_{\alpha}\mathbb{Z}_2$ has exactly two extremal tracial states. 

Conversely, assume that $A\rtimes_{\alpha}\mathbb{Z}_2$ has exactly two extremal tracial states. 
Then there exists a unitary element $V$ in $\pi_{\tau}(A)^{''}$ such that $\tilde{\alpha} =\mathrm{Ad}(V)$ by Proposition \ref{pro:strongly outer}. 
Since $\pi_{\tau}(A)^{''}$ is a factor, there exists a real number $t$ such that $V^{2}=e^{2\pi it}1$, and put $U_{\alpha} :=e^{-\pi it}V$. 
Then $U_{\alpha}$ has the desired property. 
\end{proof}

\begin{rem}\label{rem:two extremal}
Choose a unitary element $U_{\alpha}$ in $\pi_{\tau}(A)^{''}$ such that $\tilde{\alpha} =\mathrm{Ad}(U_{\alpha})$ and $U_{\alpha}^2 =1$. \\ 
(i) Let $V$ be a unitary element in $\pi_{\tau}(A)^{''}$ such that $\tilde{\alpha} =\mathrm{Ad}(V)$. 
Then there exists a real number $t$ such that $V=e^{2\pi it}U_{\alpha}$ because $\pi_{\tau}(A)^{''}$ is a factor. 
Hence the value $|\tilde{\tau} (V)|$ does not depend on the choice of such a unitary element $V$ and is determined by $\alpha$. 
We have $|\tilde{\tau} (V)|\in [0,1]$ and $|\tilde{\tau} (V)| =1$ if and only if $\tilde{\alpha} =\mathrm{id}$. 
Hence if $\alpha$ is an outer action, then $|\tilde{\tau} (V)|\in [0,1)$. \\ 
(ii) The proof above shows that two extremal tracial states $\omega_1$ and $\omega_2$ on $A\rtimes_{\alpha}\mathbb{Z}_2$ are given by the restriction of $\sigma_1$ and $\sigma_2$. 
Therefore we have 
$$
\bar{\omega}_{1} (x_0+x_1\lambda ) = \bar{\tau} (x_0) + \tilde{\tau} (x_1 U_{\alpha}) \quad \mathrm{and} \quad \bar{\omega}_{2} (x_0+x_1\lambda ) = \bar{\tau} (x_0) - \tilde{\tau} (x_1 U_{\alpha}) 
$$
where $\lambda$ is the implementing unitary of $\alpha$ and $x_0,x_1\in M(A)$. 
Note that if $W$ is a unitary element in $\pi_{\tau}(A)^{''}$ such that $\tilde{\alpha} =\mathrm{Ad}(W)$ and $W^2=1$, then $W=U_{\alpha}$ or $W=-U_{\alpha}$. 
\end{rem}

We shall define a conjugacy invariant for outer actions of $\mathbb{Z}_2$ on a simple C$^*$-algebra with a unique tracial state. 
By Proposition \ref{pro:two extremal} and Remark \ref{rem:two extremal}, we see that the following definition is well-defined. 
\begin{Def}\label{def:trace-invariant}
Let $\alpha$ be an outer action of $\mathbb{Z}_2$ on a simple C$^*$-algebra $A$ with a unique tracial state $\tau$. 
Define an invariant $\epsilon (\alpha)\in [0,1]$ of $\alpha$ by $\epsilon (\alpha) =1$ if $A\rtimes_{\alpha}\mathbb{Z}_2$ has a unique tracial state and 
$\epsilon (\alpha) = | \tilde{\tau}(V)|$ if $A\rtimes_{\alpha}\mathbb{Z}_2$ has exactly two extremal tracial states where $V$ 
is a unitary element in $\pi_{\tau}(A)^{''}$ such that $\tilde{\alpha} =\mathrm{Ad} (V)$. 
\end{Def}

\subsection{Kirchberg's central sequence C$^*$-algebras}\label{subsec:kir}
We shall recall some properties of Kirchberg's central sequence C$^*$-algebras in \cite{Kir2}. See also \cite[Section 5]{Na2}. 
For a $\sigma$-unital C$^*$-algebra $A$, set 
$$
c_0(A):=\{(a_n)_{n\in\mathbb{N}}\in \ell^{\infty}(\mathbb{N}, A)\; |\; \lim_{n \to \infty}\| a_n\| =0 \}, \; 
A^{\infty}:=\ell^{\infty}(\mathbb{N}, A)/c_0(A). 
$$
Let $B$ be a C$^*$-subalgebra of $A$. 
We identify $A$ and $B$ with the C$^*$-subalgebras of $A^\infty$ consisting of equivalence classes of 
constant sequences.  Put 
$$
A_{\infty}:=A^{\infty}\cap A^{\prime},\; \mathrm{Ann}(B,A^{\infty}):=\{(a_n)_n\in A^{\infty}\cap B^{\prime}\; |\; (a_n)_nb =0
\;\mathrm{for}\;\mathrm{any}\; b\in B \}.
$$
Then $\mathrm{Ann}(B,A^{\infty})$ is an closed two-sided ideal of $A^{\infty}\cap B^{\prime}$, and define 
$$
F(A):=A_{\infty}/\mathrm{Ann}(A,A^{\infty}).
$$
We call $F(A)$ the \textit{central sequence C$^*$-algebra} of $A$. A sequence $(a_n)_n$ is said to be 
\textit{central} if $\lim_{n\to \infty}\| a_na-aa_n \| =0$ for all $a\in A$. A central sequence 
is a representative of an element in $A_{\infty}$. 
Since $A$ is $\sigma$-unital, $A$ has a countable approximate unit $\{h_n\}_{n\in\mathbb{N}}$. 
It is easy to see that $[(h_n)_n]$ is a unit in $F(A)$. If $A$ is unital, then $F(A)=A_{\infty}$. 
Note that $F(A)$ is isomorphic to $M(A)^\infty\cap A^{\prime}/\mathrm{Ann}(A,M(A)^{\infty})$ and $F(A\otimes\mathbb{K})$. 
If $\alpha$ is an automorphism of $A$, $\alpha$ induces natural automorphisms of $A^{\infty}$, $A_{\infty}$ and $F(A)$. 
We denote them by the same symbol $\alpha$ for simplicity. 
Note that if $\alpha$ is an inner automorphism of $A$, then the induced automorphism of $\alpha$ on $F(A)$ is the identity map. 

\subsection{Strict comparison and almost stable rank one}\label{sec:strict comparison}
For positive elements $a,b\in A$, we say that $a$ is \textit{Cuntz smaller than} $b$, written $a\precsim b$, if there exists a 
sequence $\{x_n\}_{n\in\mathbb{N}}$ of $A$ such that $\| x_n^*bx_n-a\|\rightarrow 0$. 
Positive elements $a$ and $b$ are said to be \textit{Cuntz equivalent}, written $a \sim b$, if 
$a\precsim b$ and $b\precsim a$. 
For $\tau\in T(A)$, put $d_{\tau} (h)=\lim_{n\rightarrow \infty}\tau\otimes\mathrm{Tr} (h^{\frac{1}{n}})$ for 
$h\in (A\otimes \mathbb{K})_{+}$. Then $d_{\tau}$ is a dimension function. 
In this paper we say that $A$ has \textit{strict comparison} if $a,b\in (A\otimes \mathbb{K})_{+}$ with $d_{\tau}(a)< d_{\tau} (b)< \infty$ for any 
$\tau \in T(A)$, then $a\precsim b$. 
A C$^*$-algebra $A$ is said to have \textit{almost stable rank one} if for every $\sigma$-unital hereditary subalgebra 
$B\subseteq A\otimes\mathbb{K}$ we have $B\subseteq \overline{\mathrm{GL}(\widetilde{B})}$. If $A$ is unital, then $A$ has stable rank one if and only if $A$ has almost stable rank one. 
We have the following proposition. See, for example, \cite[Corollary 3.3 and Proposition 3.4]{Na2}. 
\begin{pro}\label{pro:comparison positive}
Let $A$ be a simple exact $\sigma$-unital stably projectionless C$^*$-algebra, and let $a$ and $b$ be positive elements in $A$. 
Assume that $A$ has strict comparison and almost stable rank one. 
If $d_{\tau}(a)= d_{\tau} (b)< \infty$ for any $\tau\in T(A)$, then $\overline{aA}$ is isomorphic to $\overline{bA}$ as right Hilbert $A$-modules. 
\end{pro}
We shall consider the comparison theory for projections in multiplier algebras of certain stably projectionless C$^*$-algebras. 
\begin{pro}\label{pro:comparison multiplier}
Let $A$ be a simple exact $\sigma$-unital stably projectionless C$^*$-algebra, and let $p$ and $q$ be projections in $M(A)$. 
Assume that $A$ has strict comparison and almost stable rank one. 
If $\bar{\tau}(p)=\bar{\tau}(q)< \infty$ for any $\tau\in T(A)$, then $p$ is Murray-von Neumann equivalent to $q$ in $M(A)$. 
\end{pro}
\begin{proof}
Consider right Hilbert $A$-modules $pA$ and $qA$. Since $pAp$ and $qAq$ are $\sigma$-unital, there exist positive elements $a$ and $b$ in $A$ 
such that $\overline{aA}=pA$ and $\overline{bA}=qA$. Note that $\{a^{\frac{1}{n}}\}_{n\in\mathbb{N}}$ and $\{b^{\frac{1}{n}}\}_{n\in\mathbb{N}}$ are approximate units 
for $pAp$ and $qAq$, respectively. Hence we have $\bar{\tau}(p)= d_{\tau}(a)$ and $\bar{\tau}(q)= d_{\tau}(b)$ for any $\tau\in T(A)$. 
By Proposition \ref{pro:comparison positive}, we see that $pA$ is isomorphic to $qA$ as right Hilbert $A$-modules. 
Therefore $p$ is Murray-von Neumann equivalent to $q$ in $M(A)$. 
\end{proof}

We denote by $\mathcal{Z}$ the Jiang-Su algebra constructed in \cite{JS}. 
The Jiang-Su algebra $\mathcal{Z}$ is a unital separable simple infinite-dimensional nuclear 
C$^*$-algebra whose K-theoretic invariant is isomorphic to that of complex numbers. 
A C$^*$-algebra $A$ is said to be $\mathcal{Z}$-\textit{stable} if $A$ is isomorphic to $A\otimes\mathcal{Z}$. 
If $A$ is a simple exact $\mathcal{Z}$-stable C$^*$-algebra with traces, then $A$ has strict comparison (see \cite[Theorem 4.5]{Ror}). 
Robert showed that if $A$ is a simple $\sigma$-unital $\mathcal{Z}$-stable stably projectionless C$^*$-algebra, then $A$ has almost stable rank one. 
Therefore we have the following corollary. 
\begin{cor}
Let $A$ be a simple exact $\sigma$-unital stably projectionless $\mathcal{Z}$-stable C$^*$-algebra, and let $p$ and $q$ be projections in $M(A)$. 
If $\bar{\tau}(p)=\bar{\tau}(q)< \infty$ for any $\tau\in T(A)$, then $p$ is Murray-von Neumann equivalent to $q$ in $M(A)$. 
\end{cor}

\subsection{Razak-Jacelon algebra and Robert's classification theorem}
Let $\mathcal{W}_{2}$ be the Razak-Jacelon algebra studied in \cite{J}, which has trivial K-groups and a unique tracial state and no unbounded traces. 
The Razak-Jacelon algebra $\mathcal{W}_{2}$ is constructed as an inductive limit C$^*$-algebra of Razak's building block in \cite{Raz}, that is,  
$$
A(n,m)= \left\{f\in C([0,1])\otimes M_m(\mathbb{C}) \ | \
\begin{array}{cc} 
f(0)=\mathrm{diag}(\overbrace{c,..,c}^k,0_{n}), 
f(1)=\mathrm{diag}(\overbrace{c,..,c}^{k+1}),  \\
c\in M_n(\mathbb{C})
\end{array} 
\right\}
$$
where $n$ and $m$ are natural numbers with $n|m$ and $k:=\frac{m}{n}-1$. 
Let $\mathcal{O}_{2}$ denote the Cuntz algebra generated by $2$ isometries 
$S_1$ and $S_2$. 
For every $\lambda_1,\lambda_2\in\mathbb{R}$ there exists by universality a one-parameter 
automorphism group $\alpha$ of $\mathcal{O}_2$ given by $\alpha_t (S_j)=e^{it\lambda_{j}}S_j$. 
Kishimoto and Kumjian showed that if 
$\lambda_{1}$ and $\lambda_{2}$ are all nonzero of the same sign and 
$\lambda_1$ and $\lambda_2$ generate $\mathbb{R}$ as a closed subgroup, then 
$\mathcal{O}_2\rtimes_{\alpha}\mathbb{R}$ is a simple stably projectionless C$^*$-algebra 
with unique (up to scalar multiple) densely defined lower semicontinuous trace 
in \cite{KK1} and \cite{KK2}. 
Moreover Robert \cite{Rob} showed that $\mathcal{W}_{2}\otimes \mathbb{K}$ is isomorphic to 
$\mathcal{O}_2\rtimes_{\alpha}\mathbb{R}$ for some $\lambda_1$ and $\lambda_2$. 
(See also \cite{Dean}.) 

We say that $A$ is a \textit{1-dimensional non-commutative CW  (NCCW) complex} if $A$ is a pullback C$^*$-algebra of the form 
\[\begin{CD}
      A @>\pi_2>>   E \\
             @VV\pi_1V      @VV\rho V  \\                 
      C([0,1])\otimes F  @>\delta_0\oplus\delta_1>> F \oplus F
\end{CD} \]
where $E$ and $F$ are finite-dimensional C$^*$-algebras and $\delta_i$ is the evaluation map at $i$. Razak's building block $A(n,m)$ 
is a 1-dimensional NCCW complex. 
Robert classified inductive limit C$^*$-algebras of 1-dimensional NCCW complexes with trivial $K_1$-groups 
(and C$^*$-algebras that are stably isomorphic to such inductive limits) in \cite{Rob}. 
In particular, if $A$ is a simple stably projectionless C$^*$-algebra, then the classification invariant of $A$ is the 3-tuple 
$(K_0(A), (T(A), T_1(A)), r_A)$ where $r_A$ is the pairing between $K_0(A)$ and $T(A)$ (see \cite[Proposition 6.2.3 and Corollary 6.2.4]{Rob}). 
For a homomorphism $\phi$ from $A$ to $B$, we denote by $K_0(\phi)$ and $T(\phi)$ the induced homomorphism from $K_0(A)$ to $K_0(B)$ by $\phi$ 
and the induced map from $T(B)$ to $T(A)$ by $\phi$, respectively. 
The following theorem is an immediate consequence of Robert's classification theorem (see \cite[Theorem 1.0.1, Proposition 3.1.7, Proposition 6.2.3 and Corollary 6.2.4]{Rob}). 
\begin{thm}\label{thm:Robert}
(Robert) \\
Let $A$ be a simple stably projectionless C$^*$-algebra that is stably isomorphic to one inductive limit C$^*$-algebra of 1-dimensional NCCW complexes with trivial 
$K_1$-groups. \\
(i) If $\alpha$ and $\beta$ are automorphisms of $A$, then $\alpha$ is approximately unitarily equivalent to $\beta$ if and only if $K_0(\alpha )=K_0(\beta )$ and $T(\alpha)=T(\beta)$.  \\
(ii) Let $\varphi$ be an automorphism of a countable abelian group $K_0(A)$ and $\gamma$ an automorphism of a topological cone $T(A)$ with $\gamma (T_1(A))=T_1(A)$. 
If $\varphi$ and $\gamma$ are compatible with the pairing $r_A$, then there exists an automorphism $\alpha$ of $A$ such that $K_0(\alpha) =\varphi$ and $T(\alpha)=\gamma$. 
\end{thm}

Robert's classification theorem shows that if $A$ is a simple approximately finite dimensional (AF) algebra with a unique tracial state and no unbounded traces, then 
$A\otimes\mathcal{W}_2$ is isomorphic to $\mathcal{W}_2$. 
In general, Robert conjectured that if $A$ and $B$ are separable nuclear C$^*$-algebras, then $T(A)$ is isomorphic to $T(B)$ as 
non-cancellative topological cones if and only if $A\otimes \mathcal{W}_2\otimes \mathbb{K}$ is isomorphic to $B\otimes \mathcal{W}_2\otimes \mathbb{K}$ 
(see \cite{San}). 

Note that every simple inductive limit C$^*$-algebra of 1-dimensional NCCW complexes has strict comparison and stable rank one. 
Moreover these C$^*$-algebras are $\mathcal{Z}$-stable by \cite[Corollary 9.2]{Tik}. 

\section{The Rohlin property}\label{sec:Rohlin}
In this section we introduce the Rohlin property for finite group actions on $\sigma$-unital C$^*$-algebras and study 
some basic properties of finite group actions with the Rohlin property. 

\begin{Def}
An action $\alpha$ of a finite group $G$ on a $\sigma$-unital C$^*$-algebra $A$ is said to have the \textit{Rohlin property} 
if there exists a partition of unity $\{e_{g}\}_{g\in G}\subset F(A)$ consisting of projections satisfying 
$$
\alpha_{g} (e_{h}) =e_{gh}
$$
for any $g,h\in G$. A family of projections $\{e_{g}\}_{g\in G}$ is called \textit{Rohlin projections} of $\alpha$. 
\end{Def}
If $A$ is unital, then the definition above coincides with the definition in \cite{I1}. 

\begin{ex}\label{ex:model rohlin}
Let $G$ be a finite group, and let $\mu^{G}$ be the action of $G$ on $M_{|G|^{\infty}}$ in \cite[Example 3.2]{I1}, that is, 
$$
\mu^{G}_g =\bigotimes_{n=1}^\infty \mathrm{Ad}(\lambda_g)
$$ 
for any $g\in G$ where $\lambda_g$ is the left regular representation of $G$ and we identify $B(\ell^{2} (G))$ with $M_{|G|}(\mathbb{C})$. 
Define an action $\nu^{G}$ of $G$ on $M_{|G|^{\infty}}\otimes \mathcal{W}_2\cong \mathcal{W}_2$ by 
$$
\nu^{G} :=\mu^{G}\otimes \mathrm{id}.
$$
Then $\nu^{G}$ has the Rohlin property. Indeed, let $\{h_n \}_{n\in\mathbb{N}}$ be an approximate unit for $\mathcal{W}_2$. Since $\mu^{G}$ has the 
Rohlin property (see \cite{I1}), there exist Rohlin projections $\{(p_{g,n})_n\}_{g\in G}$ of $\mu^{G}$ in $(M_{|G|^\infty})_{\infty}$. 
Put $e_{g}:=[(p_{g,n}\otimes h_n)_n]\in F(M_{|G|^{\infty}}\otimes \mathcal{W}_2)$ for any $g\in G$. Then we see that $\{e_g\}_{g\in G}$ are Rohlin 
projections of $\nu^{G}$. 
\end{ex}
The example above shows that for any finite group $G$, there exists an action of $G$ on $\mathcal{W}_2$ with the Rohlin property. 
Moreover we see that $\mathcal{W}_2\rtimes_{\nu^{G}} G$ is isomorphic to $\mathcal{W}_2$ because $M_{|G|^{\infty}}\rtimes_{\mu^{G}}G$ is a simple unital AF algebra with a 
unique tracial state. 

If $G$ is a finite abelian group, then we have the following characterization of the Rohlin property. 
\begin{pro}\label{pro:rohlin unitary}
Let $\alpha$ be an action of a finite abelian group $G$ on a $\sigma$-unital C$^*$-algebra $A$. 
Then $\alpha$ has the Rohlin property if and only if there exists a unitary representation $u$ of $\hat{G}$ on $F(A)$ such that 
$$
\alpha_{g} (u(\gamma))= \gamma (g) u (\gamma)
$$
for any $g\in G$ and $\gamma\in \hat{G}$. 
\end{pro}
\begin{proof}
Assume that $\alpha$ has the Rohlin property, and choose Rohlin projections $\{e_g \}_{g\in G}$ of $\alpha$. 
Define a map $u$ from $\hat{G}$ to $F(A)$ by 
$$
u(\gamma ):=\sum_{g\in G}\overline{\gamma (g)}e_g
$$
for any $\gamma\in \hat{G}$. Then $u$ is a unitary representation such that $\alpha_{g} (u(\gamma))= \gamma (g) u (\gamma)$ for any $g\in G$ and $\gamma\in \hat{G}$. 

Conversely, assume that there exists a unitary representation $u$ of $\hat{G}$ on $F(A)$ such that $\alpha_{g} (u(\gamma))= \gamma (g) u (\gamma)$ for any $g\in G$ and $\gamma\in \hat{G}$. 
For any $g\in G$, put 
$$
e_g:=\frac{1}{|G|}\sum_{\gamma\in\hat{G}}\gamma (g)u(\gamma).
$$
Then $\{e_g\}_{g\in G}$ are Rohlin projections of $\alpha$. 
\end{proof}

The following lemma is an analogous result of \cite[Lemma 3.3]{I1} in our case. 
\begin{lem}\label{lem:key rohlin}
Let $A$ be a $\sigma$-unital C$^*$-algebra with almost stable rank one, and let $\alpha$ and $\beta$ be actions of a finite group $G$ on $A$ with 
the Rohlin property. Assume that $\alpha_g$ and $\beta_g$ are approximately unitarily equivalent for any $g\in G$. 
Then for every finite set $F\subset A$ and positive number $\epsilon$, there exists a unitary element $v$ in $\tilde{A}$ such that the following hold for 
any $x\in F$: 
$$
\|\beta_g(x)-\mathrm{Ad}(v^*)\circ \alpha_g\circ \mathrm{Ad}(v)(x) \| <\epsilon , \quad g\in G,
$$
$$
\| [v,x] \| <\epsilon +\sup_{g\in G}\| \beta_g(x)-\alpha_g(x) \| .
$$
\end{lem}
\begin{proof}
Put 
$$
F_1:= \bigcup_{g\in G}\beta_g (F) ,
$$
and there exist unitary elements $\{v_g\}_{g\in G}$ in $M(A)$ such that 
$$
\| v_g\beta_g(y)v_g^*-\alpha_g(y) \| <\frac{\epsilon}{2} 
$$
for any $y\in F_1$. We choose Rohlin projections $\{e_g\}_{g\in G}$ of $\alpha$. 
Then there exist positive contractions $\{f_g\}_{g\in G}$ in $A_{\infty}$ such that $[f_g] =e_g$ for any $g\in G$. 
Note that we have $f_gf_ha=0$ if $g\neq h$, $f_g^2a=f_ga$ and $\alpha_g(f_h)a=f_{gh}a$ for any $a\in A$. 
Set 
$$
w:= \sum_{g\in G}v_g f_g \in A^{\infty}.
$$
For any $a\in A$, we have 
\begin{align*}
w^*w a 
& =\sum_{g\in G}\sum_{h\in G} f_g v_g^* v_hf_ha =\sum_{g\in G}\sum_{h\in G} f_g v_g^* v_h a f_h \\
& =\sum_{g\in G}\sum_{h\in G} v_g^* v_h a f_g f_h  =\sum_{g\in G}a f_g =a.
\end{align*}
Since $A\subset \overline{\mathrm{GL}(\tilde{A})}$, there exists a bounded sequence $(z_n)_n$ consisting of invertible elements in $\tilde{A}$ such that 
$w=(z_n)_n$ in $(\tilde{A})^{\infty}$. Put $u_n:= z_n(z_n^*z_n)^{-\frac{1}{2}}$ and $u:=(u_n)_n\in (\tilde{A})^{\infty}$. Then $u_n$ is a unitary element in $\tilde{A}$ 
and we have  
$$
\| (z_n-u_n)a \| =\| u_n((z_n^*z_n)^{\frac{1}{2}}a-a) \| =  \| (z_n^*z_n)^{\frac{1}{2}}a-a  \|  \rightarrow 0
$$
for any $a\in A$. 
Therefore we have $wa=ua$ in $A^{\infty}$. Note that we may not have $aw=au$. 
For $x\in F$, we have 
\begin{align*}
\| \beta_g  (x)-\mathrm{Ad}(u^*)\circ \alpha_g \circ \mathrm{Ad}(u) (x) \| 
& =\| u\beta_g(x)u^* -\alpha_g (uxu^*) \| \\
& =\| w\beta_g(x)w^* -\alpha_g (wxw^*) \| \\ 
& =\| \sum_{h\in G}v_h\beta_g(x)f_h v_h^* -\alpha_{g} (\sum_{h\in G}v_h x f_h v_h^* ) \| \\
& =\| \sum_{h\in G}f_hv_h\beta_g(x) v_h^* -\alpha_{g} (\sum_{h\in G}f_hv_h x  v_h^* ) \| \\
& =\| \sum_{h\in G}f_{gh}v_{gh}\beta_{g}(x)v_{gh}^* -\sum_{h\in G}f_{gh}\alpha_{g} (v_h x v_h^*) \| \\
& =\| \sum_{h\in G}f_{gh}(v_{gh}\beta_g(x)v_{gh}^*- \alpha_g (v_h x v_h^*)) \| .
\end{align*}
Since $f_hf_ka=0$ for any $a\in A$ if $h\neq k$, for $x\in F$, we have 
\begin{align*}
\| \beta_g & (x)-\mathrm{Ad}(u^*)\circ \alpha_g \circ \mathrm{Ad}(u) (x) \| \\ 
& \leq \sup_{h\in G}\{\| v_{gh}\beta_g(x)v_{gh}^*- \alpha_g (v_h x v_h^*) \| \}\\
& \leq \sup_{h\in G}\{\| v_{gh}\beta_{gh}(\beta_{h^{-1}}(x))v_{gh}^*- \alpha_{gh} (\beta_{h^{-1}} (x))\| +\|\alpha_{gh} (\beta_{h^{-1}} (x)) - \alpha_g (v_h x v_h^*) \|\} \\
& \leq \frac{\epsilon}{2}+\sup_{h\in G}\{\|\alpha_{h}(\beta_{h^{-1}} (x))- v_h\beta_h (\beta_{h^{-1}} (x)) v_h^* \| \} \\
& < \epsilon .
\end{align*}
For $x\in F$, we have 
\begin{align*}
\| [u,x] \| & = \| uxu^*-x\| =\| wxw^*-x\| \\
& =\| \sum_{g\in G}f_g(v_g xv_g^* -x) \| \\
& \leq \sup_{g\in G}\{\| v_g\beta_g (\beta_{g^{-1}}(x))v_g^* -\alpha_g (\beta_{g^{-1}}(x)) \| +\|\alpha_g (\beta_{g^{-1}}(x)) -x\| \} \\
& < \epsilon +\sup_{g\in G}\| \beta_g(x)-\alpha_g(x) \|.
\end{align*}
Set $v:=u_n$ for sufficiently large $n$. Then we obtain the conclusion. 
\end{proof}

We can show the following theorem by Lemma \ref{lem:key rohlin} and the Bratteli-Elliott-Evans-Kishimoto intertwining argument \cite{EK}. 
Indeed, the same proof as \cite[Theorem 3.5]{I1} works by using Lemma \ref{lem:key rohlin} instead of \cite[Lemma 3.3]{I1}. 
\begin{thm}\label{thm:Rohlin}
Let $A$ be a separable C$^*$-algebra with almost stable rank one, and let $\alpha$ and $\beta$ be actions of a finite group $G$ on $A$ with 
the Rohlin property. Assume that $\alpha_g$ and $\beta_g$ are approximately unitarily equivalent for any $g\in G$. 
Then there exists an approximately inner automorphism $\theta$ of $A$ such that 
$$
\theta \circ \alpha_g \circ \theta^{-1} =\beta_g, \quad g\in G. 
$$
\end{thm}

\begin{rem}
Izumi showed such a statement for unital separable C$^*$-algebras $A$ without assuming that $A$ has stable rank one. 
We do not know that we can show the theorem above without assuming that $A$ has almost stable rank one. 
Note that every simple separable $\mathcal{Z}$-stable stably projectionless C$^*$-algebra has almost stable rank one (see Section \ref{sec:strict comparison}). 
\end{rem}

\begin{cor}\label{cor:Razak-Rohlin}
Let $\alpha$ be an action of a finite group $G$ on $\mathcal{W}_2$ with the Rohlin property. 
Then $\alpha$ is conjugate to $\nu^{G}$ in Example \ref{ex:model rohlin}. 
\end{cor}
\begin{proof}
By Razak's classification theorem \cite{Raz}, every automorphism of $\mathcal{W}_2$ is approximately inner. 
Therefore Theorem \ref{thm:Rohlin} implies the conclusion. 
\end{proof}

We shall show that there exists an action $\alpha$ of $\mathbb{Z}_2$ on $\mathcal{W}_2$, which does not have the Rohlin property. 
This action is locally representable (see Definition \ref{def:locally}) and will be classified in Section \ref{sec:sym}. 

\begin{ex}\label{ex:non-Rohlin}
Let $\beta$ be an action of $\mathbb{Z}_2$ on $M_{2^\infty}$ defined by 
$$
\beta =\bigotimes_{n=1}^\infty\mathrm{Ad}(1_{2^n-1}\oplus (-1))\quad \mathrm{on}\quad \bigotimes_{n=1}^\infty M_{2^n}(\mathbb{C}) .
$$
Define a $\mathbb{Z}_2$-action $\alpha$ on $M_{2^\infty}\otimes \mathcal{W}_2\cong \mathcal{W}_2$ by 
$$
\alpha :=\beta\otimes \mathrm{id}.
$$
Then it is easily seen that $\beta$ is an outer action on $M_{2^{\infty}}$ but there exists a unitary element $U$ in a II$_1$ factor $\pi_{\tau} (M_{2^{\infty}})^{''}$ 
such that $\tilde{\beta } =\mathrm{Ad}(U)$ and $\tilde{\tau} (U)=0$ where $\tau$ is the unique tracial state on $M_{2^{\infty}}$. 
Hence we have $\tilde{\alpha} =\mathrm{Ad} (U\otimes 1)$, and $\mathcal{W}_2\rtimes_{\alpha} \mathbb{Z}_2$ has exactly two extremal tracial states by Proposition \ref{pro:two extremal}. 
(See also \cite[Example 2.9]{Phi}.) 
Therefore $\alpha$ is not cocycle conjugate to $\nu^{\mathbb{Z}_2}$ because $\mathcal{W}_2\rtimes_{\alpha} \mathbb{Z}_2$ is not isomorphic to $\mathcal{W}_2$. 
Corollary \ref{cor:Razak-Rohlin} implies that $\alpha$ does not have the Rohlin property. 
Note that the invariant $\epsilon(\alpha)$ defined in Definition \ref{def:trace-invariant} is equal to $0$. 
\end{ex}

Let $\alpha$ be an action of a finite group $G$ on a C$^*$-algebra $A$. Then the $K$-groups of $A$ have $G$-module structures induced by $\alpha$. 
We denote by $(K_0(A), \alpha)$ and $(K_1(A), \alpha)$ such $G$-modules. 
Izumi showed that there exist strong $K$-theoretical constraints for the Rohlin property in the unital case \cite[Theorem 3.13]{I1} and \cite[Theorem 3.3]{I2}. 
We do not use projections and unitary elements in $\tilde{A}$ for the definition of the Rohlin property. 
Nevertheless, we can show that there exist the same $K$-theoretical constraints by the following proposition. 

\begin{pro}\label{pro:K-con}
Let $A$ be a separable simple C$^*$-algebra, and let $\alpha$ be an action of a finite group $G$ on $A$ with the Rohlin property. 
Then there exists a unital separable C$^*$-algebra $B$, and an action $\beta$ of $G$ on $B$ with the Rohlin property 
such that $(K_i(A), \alpha)$ is isomorphic to $(K_i(B), \beta)$ as $G$-modules for $i=1,2$. 
\end{pro}
\begin{proof}
Let $\mathcal{O}_{\infty}$ be the Cuntz algebra generated by an infinite sequence of isometries. Then $A\otimes \mathcal{O}_{\infty}$ is purely infinite 
(see, for example, \cite[Theorem 4.1.10]{Ror1}). 
Since $\alpha\otimes \mathrm{id}$ is outer, $(A\otimes\mathcal{O}_{\infty})\rtimes_{\alpha\otimes \mathrm{id}}G$ is simple and purely infinite by \cite[Lemma 10]{KK2}. 
Therefore we see that $(A\otimes\mathcal{O}_{\infty})^{\alpha\otimes \mathrm{id}}$ is purely infinite, and hence  
$(A\otimes\mathcal{O}_{\infty})^{\alpha\otimes \mathrm{id}}$ has a nonzero projection $p$. 
Put $B:=p(A\otimes\mathcal{O}_{\infty})p$ and $\beta :=(\alpha\otimes\mathrm{id}) |_{p(A\otimes\mathcal{O}_{\infty})p}$. Choose Rohlin projections 
$\{e_g=[(f_{n,g})_n] \}_{g\in G}$ of $\alpha$ in $F(A)$ where $(f_{n,g})_n \in A_{\infty}$. Then $\{((f_{n,g}\otimes 1)p)_n\}_{g\in G}$ are Rohlin projections of $\beta$, and hence 
$\beta$ has the Rohlin property. 
Since the inclusion map $i$ from $p(A\otimes\mathcal{O}_{\infty})p$ into $A\otimes\mathcal{O}_{\infty}$ induces isomorphisms of $K$-groups by Brown's theorem \cite{B} 
and $\mathcal{O}_{\infty}$ is $KK$-equivalent to $\mathbb{C}$, we see that $(K_i(B), \beta)$ is isomorphic to $(K_i(A), \alpha)$ as $G$-modules for $i=1,2$. 
\end{proof}

We refer the reader to \cite{I2} for details of $K$-theoretical constraints. 
The following corollary is an immediate consequence of Proposition \ref{pro:K-con} and \cite[Theorem 3.6]{I2}. 

\begin{cor}
Let $A$ be a simple separable C$^*$-algebra such that either $K_0(A)$ or $K_1(A)$ is isomorphic to $\mathbb{Z}$. 
Then there exist no non-trivial finite group actions on $A$ with the Rohlin property. 
\end{cor} 
For given $K$-theoretical invariants, Izumi constructed model actions with the Rohlin property on Kirchberg algebras \cite[Theorem 5.3]{I2}. 
As an application of this construction, we obtain the following corollary by a similar argument as in the proof of Proposition \ref{pro:K-con} and \cite[Corollary 5.5]{I2}. 
\begin{cor}
Let $\alpha$ be an action of a finite group $G$ with the Rohlin property on a simple separable nuclear C$^*$-algebra $A$ in the UCT class. 
Then $A\rtimes_{\alpha}G$ is in the UCT class. 
\end{cor}

\section{Approximately representable actions}\label{sec:app-rep}
In this section we introduce the approximate representability for finite group actions on $\sigma$-unital C$^*$-algebras 
and study some basic properties of approximately representable actions. Moreover we give some examples of Rohlin actions on $\mathcal{W}_2$ 
by using a classification result of some approximately representable actions. 

\begin{Def}\label{def:app rep}
An action $\alpha$ of a finite abelian group $G$ on a $\sigma$-unital C$^*$-algebra $A$ is said to be \textit{approximately representable} if 
there exist elements $\{w(g)\}_{g\in G}$ in $(A^{\alpha})_{\infty}$ such that a map $u$ from $G$ to $F(A^{\alpha})$ given by $u(g)=[w(g)]$ is a unitary representation of $G$ and 
$$
\alpha_g(a) =w(g) a w(g)^* 
$$  
in $A^{\infty}$ for any $a\in A$ and $g\in G$. 
\end{Def}
If $A$ is unital, then the definition above coincides with the definition in \cite{I1}. 
\begin{rem}
For a general finite group $G$, we can define the approximate representability as in \cite[Remark 3.7]{I1}. 
Put $F(A^{\alpha}, A):= A^{\infty}\cap (A^{\alpha})^\prime /\mathrm{Ann}(A^{\alpha},A^{\infty})$ (see Section \ref{subsec:kir}). 
An action $\alpha$ of a finite group $G$ on a $\sigma$-unital C$^*$-algebra $A$ is said to be approximately representable 
if there exist elements $\{w(g)\}_{g\in G}$ in $A^{\infty}\cap (A^{\alpha})^{\prime}$ such that a map from $G$ to $F(A^{\alpha}, A)$ given by $u(g)=[w(g)]$ is a unitary representation of $G$ and 
$$
\alpha_g(a) =w(g) a w(g)^*, \quad \mathrm{in} \; A^{\infty}, \quad a\in A,\; g\in G 
$$  
$$
\alpha_g(u(h))=u(g h g^{-1}), \quad \mathrm{in} \; F(A^{\alpha}, A), \quad g,h\in G. 
$$ 
\end{rem}

We shall show that the approximate representability is the "dual" notion of the Rohlin property. 
Note that if $\alpha$ is an action of a finite group $G$ on $A$, then we have $(A^{\alpha})^{\infty}=(A^{\infty})^{\alpha}$. 
In a similar way, we see the following lemma. 

\begin{lem}\label{lem:dual}
Let $\alpha$ be an action of a finite group $G$ on $A$, and let $B$ be a subalgebra of $A^{\alpha}$. 
If $w$ is an element in $A^{\infty}\cap B^{\prime}$ such that $wa=\alpha_g(w)a$ for any $a\in A$ and $g\in G$, 
then there exists an element $y\in (A^{\alpha})^{\infty}\cap B^{\prime}$ such that $wa=ya$ in $A^{\infty}$ for any $a\in A$. 
\end{lem}
\begin{proof}
Let $(w_n)_n$ be a representative of $w$. Define $y_n:=\frac{1}{|G|}\sum_{g\in G}\alpha_g(x_n) \in A^{\alpha}$, and put $y=(y_n)_n$. 
Then $y$ has the desired property. 
\end{proof}

See \cite[Lemma 3.8]{I1} for the unital case of the following proposition. 

\begin{pro}\label{thm:dual}
Let $\alpha$ be an action of a finite abelian group $G$ on a $\sigma$-unital C$^*$-algebra $A$. Then \\
(i) $\alpha$ has the Rohlin property if and only if $\hat{\alpha}$ is approximately representable; \\
(ii) $\alpha$ is approximately representable if and only $\hat{\alpha}$ has the Rohlin property. 
\end{pro}
\begin{proof}
(i) Assume that $\alpha$ has the Rohlin property. Then there exists a unitary representation $u$ of $\hat{G}$ on $F(A)$ such that 
$\alpha_{g} (u(\gamma))= \gamma (g) u (\gamma)$ for any $g\in G$ and $\gamma\in \hat{G}$ by Proposition \ref{pro:rohlin unitary}. 
Choose elements $\{w (\gamma )\}_{\gamma\in \hat{G}}$ in $A_{\infty}$ such that $[w(\gamma )]=u(\gamma )$, and let $\lambda_g$  be the implementing 
unitary of $\alpha_{g}$ in $M(A\rtimes_{\alpha} G)$. 
Then we have 
$$
a\lambda_g w(\gamma)\lambda_{g}^* =a\gamma (g)w(\gamma)
$$ 
for any $a\in A$, $g\in G$ and $\gamma\in\hat{G}$. This equation implies 
$$
w({\gamma})^* a\lambda_{g}w(\gamma ) =\gamma (g)a\lambda_g, 
$$
and hence 
$$
\hat{\alpha}_{\gamma}(x) =w(\gamma )^*x w(\gamma )
$$
for any $x\in A\rtimes_{\alpha} G$ and $\gamma\in \hat{G}$. Since we have $(A\rtimes_{\alpha}G)^{\hat{\alpha}}=A$, $\hat{\alpha}$ is approximately representable. 
The converse follows from a similar computation. 

(ii) Assume that $\alpha$ is approximately representable, and take elements $\{w(g)\}_{g\in G}$ in $(A^{\alpha})_{\infty}$ as in Definition \ref{def:app rep}. 
Then we have 
$$
\lambda_g a \lambda_g^*=w(g)aw(g)^* \quad a\in A, g\in G
$$ 
and 
$$
\lambda_g w(h) \lambda_g^* = w(h) \quad g,h\in G
$$
where $\lambda_g$ is the implementing unitary of $\alpha_{g}$ in $M(A\rtimes_{\alpha} G)$. 
Put $z(g):=w(g)^*\lambda_g\in (A\rtimes_{\alpha}G)^{\infty}$ for any $g\in G$. By the equations above, we have 
$$
a\lambda_h z(g) =a\lambda_h w(g)^*\lambda_g =aw(g)^*\lambda_h\lambda_g =aw(g)^*\lambda_g\lambda_h=w(g)^*\lambda_g a\lambda_h=z(g)a\lambda_h
$$
for any $a\in A$ and $g,h\in G$. Therefore $\{z_g\}_{g\in G}$ are elements in $(A\rtimes_{\alpha}G)_{\infty}$. 
Define a map $v$ from $G$ to $F(A\rtimes_{\alpha}G)$ by $v(g)=[z(g)]$. 
Then it can be easily checked that $v$ is a unitary representation and $\hat{\alpha}_{\gamma}(v(g))=\gamma(g)v(g)$ for any $\gamma\in \hat{G}$ and $g\in G$. 
Proposition \ref{pro:rohlin unitary} implies that $\hat{\alpha}$ has the Rohlin property. 

Conversely, assume that $\hat{\alpha}$ has the Rohlin property. Then there exists a unitary representation $u$ of $G$ on $F(A\rtimes_{\alpha}G)$ such that 
$\hat{\alpha}_{\gamma}(u(g))=\gamma (g) u(g)$ for any $\gamma\in \hat{G}$ and $g\in G$ by Proposition \ref{pro:rohlin unitary}. 
Choose elements $\{z(g) \}_{g\in G}$ in $(A\rtimes_{\alpha}G)_{\infty}$ such that $[z(g)]=u(g)$. Then we have 
$$ 
\hat{\alpha}_{\gamma} (\lambda_g z(g)^*)x= \gamma (g)\lambda_g \overline{\gamma (g)} z(g)^*x= \lambda_g z(g)^*x 
$$
for any $x\in A\rtimes_{\alpha}G$, $g\in G$ and $\gamma\in\hat{G}$ where $\lambda_g$ is the implementing unitary of $\alpha_{g}$ in $M(A\rtimes_{\alpha} G)$. 
Since we have $\lambda_g z(g)^*b =b\lambda_g z(g)^*$ for any $b\in A^{\alpha}$ and $(A\rtimes_{\alpha}G)^{\hat{\alpha}}=A$, Lemma \ref{lem:dual} implies that there exist elements 
$\{y(g) \}_{g\in G}$ in $A^{\infty}\cap (A^{\alpha})^{\prime}$ such that 
$y(g)x=\lambda_g z(g)^* x$ for any $x\in A\rtimes_{\alpha}G$ and $g\in G$. 
For any $a\in A$ and $g,h\in G$, we have 
$$
\alpha_h (y(g))a =\lambda_h\lambda_g z(g)^*\lambda_h^*a =\lambda_h\lambda_g\lambda_h^*a z(g)^*=\lambda_g z(g)^* a=y(g)a. 
$$
Hence there exist elements $\{w(g) \}_{g\in G}$ in $(A^{\alpha})_{\infty}$ such that $w(g)a=\lambda_g z(g)^*a$ for any $a\in A$ and $g\in G$ by Lemma \ref{lem:dual}. 
It can be easily checked that a map $v$ from $G$ to $F(A^{\alpha})$ given by $v(g)=[w(g)]$ is a unitary representation of $G$ and $\alpha_g(a)=w(g)aw(g)^*$ for any 
$a\in A$ and $g\in G$. Consequently $\alpha$ is approximately representable. 
\end{proof}

By Theorem \ref{thm:Rohlin}, Proposition \ref{thm:dual} and the Takesaki-Takai duality, we can show the following classification result of some approximately representable actions. 
See \cite[Corollary 3.9]{I1} for the unital case. 
\begin{cor}\label{cor:app-classification}
Let $\alpha$ and $\beta$ be actions of a finite abelian group $G$ on a separable C$^*$-algebra $A$, and let $e_{\alpha}$ and $e_{\beta}$ be the projections in $M(A\rtimes_{\alpha}G)$ and 
$M(A\rtimes_{\beta}G)$, respectively, defined in Section \ref{sec:crossed}. 
Assume that $\alpha$ and $\beta$ are approximately representable and $A\rtimes_{\alpha}G$ and $A\rtimes_{\beta}G$ have almost stable rank one. 
Then $\alpha$ and $\beta$ are conjugate if and only if there exists an isomorphism $\theta$ from $A\rtimes_{\alpha}G$ onto $A\rtimes_{\beta}G$ such that 
$\theta \circ \hat{\alpha}_{\gamma} \circ \theta^{-1}$ is approximately unitarily equivalent to $\hat{\beta}_{\gamma}$ for any $\gamma\in\hat{G}$ and 
$\bar{\theta} (e_{\alpha})$ is Murray-von Neumann equivalent to $e_{\beta}$ (in $M(A\rtimes_{\beta}G)$). 
\end{cor}
\begin{proof}
The only if part is obvious. We will show the if part. By Theorem \ref{thm:Rohlin} and Proposition \ref{thm:dual}, there exists an approximately inner automorphism $\theta_1$ of 
$A\rtimes_{\beta}G$ such that $\theta_1\circ \theta \circ \hat{\alpha}_{\gamma} \circ \theta^{-1} \circ \theta_1^{-1} =\hat{\beta}_{\gamma}$ for any $\gamma\in\hat{G}$. 
Replacing $\theta$ with $\theta_1\circ \theta$,  we may assume that $\theta$ is an isomorphism from $A\rtimes_{\alpha}G$ onto $A\rtimes_{\beta}G$ such that 
$\theta\circ \hat{\alpha}_{\gamma} =\hat{\beta}_{\gamma}\circ \theta$ for any $\gamma$ and $\bar{\theta} (e_{\alpha})$ is Murray-von Neumann equivalent to $e_{\beta}$ 
because $\theta_1$ is approximately inner. Then there exists an isomorphism $\Psi$ from $A\rtimes_{\alpha}G\rtimes_{\hat{\alpha}}\hat{G}$ onto $A\rtimes_{\beta}G\rtimes_{\hat{\beta}}\hat{G}$ 
such that $\Psi \circ \hat{\hat{\alpha}}_g =\hat{\hat{\beta}}_g\circ \Psi$ for any $g\in G$ and $\bar{\Psi}(e_{\alpha})$ is Murray-von Neumann equivalent to $e_{\beta}$ in 
$M(A\rtimes_{\beta}G\rtimes_{\hat{\beta}}\hat{G})^{\bar{\hat{\hat{\beta}}}}$. 
Hence the Takesaki-Takai duality theorem implies that 
there exists an automorphism $\Theta$ of $A\otimes K(\ell^2(G))$ such that $\Theta$ intertwines $\alpha\otimes \mathrm{Ad}(\rho )$ and 
$\beta\otimes \mathrm{Ad}(\rho )$. Let $e\in K(\ell^2(G))$ be the projection onto the constant functions. 
Then we have 
\begin{align*}
(A,\alpha) & \cong ((1\otimes e)(A\otimes K(\ell^2(G))) (1\otimes e), \alpha\otimes \mathrm{Ad}(\rho )|_{(1\otimes e)(A\otimes K(\ell^2(G))) (1\otimes e)} ) \\
           & \cong (\bar{\Theta}(1\otimes e)(A\otimes K(\ell^2(G))) \bar{\Theta}(1\otimes e),  \beta\otimes \mathrm{Ad}(\rho )|_{\bar{\Theta}(1\otimes e)(A\otimes K(\ell^2(G))) \bar{\Theta}(1\otimes e)} ).
\end{align*}
Since $\bar{\Psi}(e_{\alpha})$ is Murray-von Neumann equivalent to $e_{\beta}$ in $M(A\rtimes_{\beta}G\rtimes_{\hat{\beta}}\hat{G})^{\bar{\hat{\hat{\beta}}}}$, 
we see that $\bar{\Theta} (1\otimes e)$ is Murray-von Neumann equivalent to $1\otimes e$ in $M(A\otimes K(\ell^2(G)))^{\overline{\beta\otimes \mathrm{Ad} (\rho )}}$ 
(see Section \ref{sec:crossed}). 
Therefore we have 
\begin{align*}
(A,\alpha) & \cong (\bar{\Theta}(1\otimes e)(A\otimes K(\ell^2(G))) \bar{\Theta}(1\otimes e),  \beta\otimes \mathrm{Ad}(\rho )|_{\bar{\Theta}(1\otimes e)(A\otimes K(\ell^2(G))) \bar{\Theta}(1\otimes e)} ) \\
& \cong  ((1\otimes e)(A\otimes K(\ell^2(G))) (1\otimes e), \beta\otimes \mathrm{Ad}(\rho )|_{(1\otimes e)(A\otimes K(\ell^2(G))) (1\otimes e)} ) \\
& \cong (A, \beta ). 
\end{align*}
\end{proof}

\begin{rem}
Let $s$ be a strictly positive element in $A^{\alpha}$. Define a positive element $s_{\alpha}$ in $A\rtimes_{\alpha}G$ by 
$$
s_{\alpha}:=\frac{1}{|G|}\sum_{g\in G}s\lambda_g 
$$
where $\lambda_g$ is the implementing unitary of $\alpha_g$ in $M(A\rtimes_{\alpha}G)$. 
It can be checked that $\overline{s_{\alpha}(A\rtimes_{\alpha}G)}=e_{\alpha}(A\rtimes_{\alpha}G)$ and the Cuntz class $[s_{\alpha}]$ does not depend on the choice of $s$. 
Since $A\rtimes_{\alpha}G$ and $A\rtimes_{\beta}G$ have almost stable rank one,   
we see that the condition that $\bar{\theta} (e_{\alpha})$ is Murray-von Neumann equivalent to $e_{\beta}$ is equivalent to the condition that $\theta (s_{\alpha})$ is Cuntz equivalent to $s_{\beta}$ 
by \cite[Theorem 3]{CEI}. 
\end{rem}

We shall consider examples of approximately representable actions. 
\begin{Def}\label{def:locally}
Let $\mathcal{C}$ be a class of C$^*$-algebras. 
An action $\alpha$ of a finite group $G$ on a C$^*$-algebra $A$ is said to be \textit{locally representable} if 
there exists an $\alpha$-invariant increasing sequence of C$^*$-subalgebras $\{A_n\}_{n\in\mathbb{N}}$ whose union is dense in $A$, and a unitary representation $u^{(n)}$ of $G$ in $M(A_n)$  
such that  
$$
\alpha_g(a) =\mathrm{Ad}(u^{(n)}(g)) (a) 
$$
for any $a\in A_n$ and $g\in G$. 
Moreover if $A_n$ is contained in $\mathcal{C}$ for any $n\in\mathbb{N}$, we say that $\alpha$ is \textit{locally representable in $\mathcal{C}$}. 
\end{Def}

Note that we do not assume that $A_n\subset A_{n+1}$ is a nondegenerate inclusion in the definition above. 
Hence $M(A_n)$ may not be contained in $M(A_{n+1})$. 
We denote by $\mathcal{C}_{\mathrm{F}}$ the class of all finite dimensional C$^*$-algebras. 
We shall show that every locally representable action on a separable C$^*$-algebra is approximately representable. 

\begin{lem}\label{lem:app-unit}
Let $\alpha$ be an action of a finite group $G$ on a separable C$^*$-algebra $A$. 
Assume that there exists an $\alpha$-invariant increasing sequence of C$^*$-subalgebras $\{A_n\}_{n\in\mathbb{N}}$ whose union is dense in $A$. 
Then there exists an approximate unit $\{h_n\}_{n\in\mathbb{N}}$ for $A$ such that $h_n\in (A_n)^\alpha$ for any $n\in\mathbb{N}$. 
\end{lem}
\begin{proof}
Let $\{F_{n}\}_{n\in\mathbb{N}}$ be an increasing sequence of finite subsets in $A$ such that $F_n\subset A_n$ and $\cup_{n=1}^\infty F_n$ is dense in $A$. 
Since $(A_n)^{\alpha}\subset A_n$ is a nondegenerate inclusion, there exists a positive contraction $h_n$ in $(A_n)^{\alpha}$ such that 
$$
\| h_n x-x \| <\frac{1}{n}
$$
for any $x\in F_n$. Then $\{h_n\}_{n\in\mathbb{N}}$ is an approximate unit for $A$. 
\end{proof}

\begin{pro}\label{pro:loc-app}
If $\alpha$ is a locally representable action of a finite group $G$ on a separable C$^*$-algebra $A$, then $\alpha$ is approximately representable. 
\end{pro}
\begin{proof}
Let $\{A_n\}_{n\in\mathbb{N}}$ be an increasing sequence of C$^*$-subalgebras of $A$ and $\{u^{(n)}\}_{n\in\mathbb{N}}$ unitary representations as in Definition \ref{def:locally}. 
There exists an approximate unit $\{h_n\}_{n\in\mathbb{N}}$ for $A$ such that $h_n\in (A_n)^\alpha$ for any $n\in\mathbb{N}$ by Lemma \ref{lem:app-unit}. 
For any $g\in G$, put $w(g):=(u^{(n)}(g)h_n)_n \in A^{\infty}$. Then we have 
$$
\alpha_g(a) =w(g) a w(g)^*
$$ 
in $A^{\infty}$ for any $a\in A$ and $g\in G$, and hence $w(g)\in A^{\infty}\cap (A^{\alpha})^{\prime}$. 
Since we have $u^{(n)}(g)h_n=h_nu^{(n)}(g)$ for any $n\in\mathbb{N}$ and $g\in G$, it can be easily checked that a map $v$ from $G$ to $F(A^{\alpha}, A)$ given by $v(g)=[w(g)]$ is a unitary representation of $G$ and 
$$
\alpha_g (v(h))= v(ghg^{-1})
$$ 
for any $g,h\in G$. 
Therefore $\alpha$ is approximately representable. 
\end{proof}

Izumi showed that if $\beta$ is an outer action of a finite group $G$ on a separable simple unital nuclear C$^*$-algebra $A$, 
then the action $\beta\otimes \mathrm{id}$ of $G$ on $A\otimes\mathcal{O}_2\cong \mathcal{O}_2$ has the Rohlin property \cite[Corollary 4.3]{I1}. 
This result can be regarded as an equivariant version of Kirchberg-Phillips's theorem \cite[Theorem 3.8]{KP}. 
We cannot show such a statement for $\mathcal{W}_2$. 
Indeed, Example \ref{ex:non-Rohlin} shows that there exists an outer action $\beta$ on $M_{2^{\infty}}$ such that the action $\beta\otimes \mathrm{id}$ on 
$M_{2^{\infty}}\otimes\mathcal{W}_2\cong\mathcal{W}_2$ does not have the Rohlin property. Note that this action is not strongly outer and 
$(M_{2^{\infty}}\otimes\mathcal{W}_2)\rtimes_{\beta\otimes \mathrm{id}} \mathbb{Z}_2$ has two extremal tracial states. 
Hence it is natural to consider the following question.

\begin{que}
Let $A$ be a separable simple nuclear C$^*$-algebra with a unique tracial state and no unbounded traces, and let $\beta$ be a strongly outer action of a finite group $G$ on $A$. 
Does the action $\beta\otimes \mathrm{id}$ of $G$ on $A\otimes\mathcal{W}_2$ have the Rohlin property?
\end{que}

We shall show the following (non-trivial) partial result of the question above as an application of Corollary \ref{cor:app-classification}. 

\begin{cor}\label{thm:rohlin-example}
Let $A$ be a simple AF algebra with a unique tracial state and no unbounded traces, and let $\beta$ be a strongly outer action of a finite abelian group $G$ on $A$. 
Assume that $\beta$ is approximately representable and $A\rtimes_{\beta}G$ is an AF algebra. 
Then 
$$
(A\otimes\mathcal{W}_2,\beta\otimes \mathrm{id}) \cong (\mathcal{W}_2, \nu^{G})
$$ 
where $\nu^{G}$ is the action in Example \ref{ex:model rohlin}, and hence $\beta\otimes \mathrm{id}$ has the Rohlin property. 
\end{cor}
\begin{proof}
 Proposition \ref{pro:strongly outer} and \cite[Theorem 3.1]{K} imply $A\rtimes_{\beta} G$ is a simple C$^*$-algebra with a unique tracial state. 
It is easy to see that $A\rtimes_{\beta} G$ has no unbounded traces. 
Since $(A\otimes\mathcal{W}_2)\rtimes_{\beta\otimes \mathrm{id}}G$ is isomorphic to $(A\rtimes_{\beta} G)\otimes \mathcal{W}_2$, 
$(A\otimes\mathcal{W}_2)\rtimes_{\beta\otimes \mathrm{id}}G$ is isomorphic to $\mathcal{W}_2$. 
On the other hand, $\mathcal{W}_2\rtimes_{\nu^G}G$ is isomorphic to $\mathcal{W}_2$ and Proposition \ref{pro:loc-app} implies that $\nu^{G}$ is approximately representable. 
Let $\theta$ be an isomorphism from $\mathcal{W}_2\rtimes_{\nu^G}G$ onto $(A\otimes\mathcal{W}_2)\rtimes_{\beta\otimes \mathrm{id}}G$. 
We denote by $\omega_{\beta\otimes\mathrm{id}}$ and $\omega_{\nu}$ the unique tracial state on $(A\otimes\mathcal{W}_2)\rtimes_{\beta\otimes \mathrm{id}}G$ and 
the unique tracial state on $\mathcal{W}_2\rtimes_{\nu^G}G$, respectively. 
Then we have 
$$
\bar{\omega}_{\beta\otimes\mathrm{id}} (e_{\alpha\otimes\mathrm{id}})=\bar{\omega}_{\nu}(e_{\nu^{G}}) =  \frac{1}{|G|}. 
$$
Hence we see that $\theta (e_{\nu^{G}})$ is Murray-von Neumann equivalent to $e_{\alpha\otimes\mathrm{id}}$ by Proposition \ref{pro:comparison multiplier}. 
Since every automorphism of $\mathcal{W}_2$ is approximately inner, we obtain the conclusion by Corollary \ref{cor:app-classification}. 
\end{proof}

\begin{rem}
(i) If an action $\beta$ on a C$^*$-algebra $A$ is locally representable in $\mathcal{C}_{\mathrm{F}}$, then $A\rtimes_{\beta}G$ is an AF algebra. \\
(ii) There exists a strongly outer action $\beta$ of a finite abelian group on $M_{2^\infty}$ such that $\beta$ is locally representable in $\mathcal{C}_{\mathrm{F}}$ and 
$\beta$ does not have the Rohlin property. Indeed, define an action $\beta$ of $\mathbb{Z}_2$ on $M_{2^\infty}$ by 
$$
\beta =\bigotimes_{n=1}^\infty\mathrm{Ad}(1_{2^{n-1}+1}\oplus -1_{2^{n-1}-1})\quad \mathrm{on}\quad \bigotimes_{n=1}^\infty M_{2^n}(\mathbb{C}) .
$$
Then $\beta$ is such an action (see \cite[Example 3.7]{Phi} for the proof). But the Corollary above shows that the action $\beta\otimes\mathrm{id}$ on 
$M_{2^{\infty}}\otimes\mathcal{W}_2\cong\mathcal{W}_2$ has the Rohlin property. 
\end{rem}

\section{Symmetries of $\mathcal{W}_2$}\label{sec:sym}
Let $\mathcal{C}_{\mathrm{R}}$ be the class of C$^*$-algebras $A$ in Robert's classification theorem \cite[Theorem 1.0.1]{Rob}, that is, 
$A$ is either a 1-dimensional NCCW complex with trivial $K_1$-group, an inductive limit C$^*$-algebra of such C$^*$-algebras or a C$^*$-algebra stably isomorphic to 
one such inductive limit. 
It is obvious that the outer actions on $\mathcal{W}_2$ in Example \ref{ex:model rohlin} and Example \ref{ex:non-Rohlin} are locally representable in $\mathcal{C}_{\mathrm{R}}$. 
In this section we shall classify the locally representable outer $\mathbb{Z}_2$-actions $(\mathcal{W}_2,\alpha)$ in $\mathcal{C}_{\mathrm{R}}$ 
up to conjugacy and cocycle conjugacy. We denote by $\tau$ the unique tracial state on $\mathcal{W}_2$ in this section. 

\begin{lem}\label{lem:closed-class}
Let $A$ be a simple stably projectionless C$^*$-algebra with no unbounded traces, and let $\alpha$ be an outer action of a finite group $G$ on $A$. 
Assume that $\alpha$ is locally representable in $\mathcal{C}_{\mathrm{R}}$. 
Then $A\rtimes_{\alpha}G$ is a simple stably projectionless C$^*$-algebra with no unbounded traces and is contained in $\mathcal{C}_{\mathrm{R}}$. 
\end{lem}
\begin{proof}
It is easy to see that $A\rtimes_{\alpha}G$ has no unbounded traces and $A^{\alpha}$ is stably projectionless. 
Since $\alpha$ is outer, $A\rtimes_{\alpha}G$ is simple by \cite[Theorem 3.1]{K}. 
Hence we see that $A\rtimes_{\alpha}G$ is a simple stably projectionless C$^*$-algebra with no unbounded traces. 
Let $\{A_n\}_{n\in\mathbb{N}}$ be an increasing sequence of C$^*$-subalgebras of $A$ and $\{u^{(n)}\}_{n\in\mathbb{N}}$ unitary representations as in Definition \ref{def:locally}. 
Then we have 
$$
A\rtimes_{\alpha}G =\overline{\bigcup_{n=1}^\infty A_n\rtimes_{\mathrm{Ad}(u^{(n)})}G}.
$$
Since $A_n\rtimes_{\mathrm{Ad}(u^{(n)})}G$ is isomorphic to the tensor product of $A_n$ and a finite dimensional C$^*$-algebra, we see that $A\rtimes_{\alpha}G$ is contained in $\mathcal{C}_{\mathrm{R}}$. 
\end{proof}

\begin{thm}\label{thm:Main-sym}
Let $\alpha$ and $\beta$ be outer actions of $\mathbb{Z}_2$ on $\mathcal{W}_2$. Assume that $\alpha$ and $\beta$ are locally representable in $\mathcal{C}_{\mathrm{R}}$. 
Then $\alpha$ is conjugate to $\beta$ if and only if $\mathcal{W}_2\rtimes_{\alpha}\mathbb{Z}_2$ is isomorphic to $\mathcal{W}_2\rtimes_{\beta}\mathbb{Z}_2$ and 
$\epsilon (\alpha)= \epsilon (\beta)$ where $\epsilon (\alpha)$ is the invariant defined in Definition \ref{def:trace-invariant}. 
\end{thm}
\begin{proof}
It is easy to see that if $\alpha$ is conjugate to $\beta$, then $\mathcal{W}_2\rtimes_{\alpha}\mathbb{Z}_2$ is isomorphic to $\mathcal{W}_2\rtimes_{\beta}\mathbb{Z}_2$ and 
$\epsilon (\alpha)= \epsilon (\beta)$. We will show the converse. Let $\theta$ be an isomorphism from $\mathcal{W}_2\rtimes_{\alpha}\mathbb{Z}_2$ onto $\mathcal{W}_2\rtimes_{\beta}\mathbb{Z}_2$. 

First, we shall show $K_0(\hat{\alpha})=-\mathrm{id}$. Put $A:= \mathcal{W}_2\rtimes_{\alpha}\mathbb{Z}_2$, and 
let $i_1$ and $i_2$ denote the inclusion from $A$ into $A\rtimes_{\hat{\alpha}}\mathbb{Z}_2$ and the inclusion from $A\rtimes_{\hat{\alpha}}\mathbb{Z}_2$ 
into $A\rtimes_{\hat{\alpha}}\mathbb{Z}_2\rtimes_{\hat{\hat{\alpha}}}\mathbb{Z}_2\cong A\otimes M_2(\mathbb{C})$, respectively. 
Since $A\rtimes_{\hat{\alpha}}\mathbb{Z}_2$ is isomorphic to $\mathcal{W}_2\otimes M_{2}(\mathbb{C})$ and $K_0(\mathcal{W}_2)=0$, we have $K_0 (i_2\circ i_1)=0$. 
The Takesaki-Takai duality implies 
$$
i_2\circ i_1(a)=
\left(\begin{array}{cc}
            a    &     0         \\ 
            0    &  \hat{\alpha} (a)   
 \end{array} \right) 
$$
for any $a\in A$ (see Section \ref{sec:crossed}). Hence we have $K_0(\hat{\alpha})=-\mathrm{id}$. Note that the same argument shows  $K_0(\hat{\beta})=-\mathrm{id}$.

(i) The case where $\mathcal{W}_2\rtimes_{\alpha}\mathbb{Z}_2$ and $\mathcal{W}_2\rtimes_{\beta}\mathbb{Z}_2$ have a unique tracial state. \\
We denote by $\omega_{\alpha}$ and $\omega_{\beta}$ the unique tracial state on $\mathcal{W}_2\rtimes_{\alpha}\mathbb{Z}_2$ and 
the unique tracial state on $\mathcal{W}_2\rtimes_{\beta}\mathbb{Z}_2$, respectively.  
We see that $\hat{\beta}$ is approximately unitarily equivalent to $\theta\circ \hat{\alpha} \circ \theta^{-1}$ by Lemma \ref{lem:closed-class} and 
Robert's classification theorem (Theorem \ref{thm:Robert}). 
Moreover we have 
$$
\bar{\omega}_{\alpha}(e_{\alpha}) = \bar{\omega}_{\beta} (e_{\beta}) =\frac{1}{2},
$$
and hence $\theta (e_{\alpha})$ is Murray-von Neumann equivalent to $e_{\beta}$ by Proposition \ref{pro:comparison multiplier}. 
Therefore Corollary \ref{cor:app-classification} shows that $\alpha$ is conjugate to $\beta$. 

(ii) The case where $\mathcal{W}_2\rtimes_{\alpha}\mathbb{Z}_2$ and $\mathcal{W}_2\rtimes_{\beta}\mathbb{Z}_2$ have two extremal tracial states. \\
Let $\omega_{\alpha ,0}$ and $\omega_{\alpha ,1}$ be extremal tracial states on $\mathcal{W}_2\rtimes_{\alpha}\mathbb{Z}_2$ 
and $\omega_{\beta ,0}$ and $\omega_{\beta ,1}$ extremal tracial states on  $\mathcal{W}_2\rtimes_{\beta}\mathbb{Z}_2$. 
Since $\mathcal{W}_2\rtimes_{\alpha}\mathbb{Z}_2\rtimes_{\hat{\alpha}}\mathbb{Z}_2$ and $\mathcal{W}_2\rtimes_{\beta}\mathbb{Z}_2\rtimes_{\hat{\beta}}\mathbb{Z}_2$ 
have a unique tracial state, we have $\omega_{\alpha ,0}\circ \hat{\alpha} =\omega_{\alpha ,1}$ and $\omega_{\beta ,0}\circ \hat{\beta} =\omega_{\beta ,1}$. 
Hence we see that $\hat{\beta}$ is approximately unitarily equivalent to $\theta\circ \hat{\alpha} \circ \theta^{-1}$ by Lemma \ref{lem:closed-class} and 
Robert's classification theorem (Theorem \ref{thm:Robert}). 
Choose a unitary element $U_{\alpha}$ in $\pi_{\tau}(\mathcal{W}_2)^{''}$ such that $\tilde{\alpha}=\mathrm{Ad} (U_{\alpha})$ and $U_{\alpha}^2=1$ 
(see Proposition \ref{pro:two extremal} and Remark \ref{rem:two extremal}). 
By Remark \ref{rem:two extremal}, we have either
$$
(\bar{\omega}_{\alpha ,0} (e_{\alpha}), \bar{\omega}_{\alpha ,1} (e_{\alpha}))=(\frac{1}{2}(1+\tilde{\tau}(U_{\alpha})), \frac{1}{2}(1-\tilde{\tau}(U_{\alpha})))
$$ 
or 
$$
(\bar{\omega}_{\alpha ,0} (e_{\alpha}), \bar{\omega}_{\alpha ,1} (e_{\alpha}))=(\frac{1}{2}(1-\tilde{\tau}(U_{\alpha})), \frac{1}{2}(1+\tilde{\tau}(U_{\alpha}))).
$$
In the same way, we can compute $\bar{\omega}_{\beta ,0}(e_{\beta})$ and $\bar{\omega}_{\beta ,1}(e_{\beta})$. 
Since $\epsilon (\alpha)= \epsilon (\beta)$, we have either 
$$
\bar{\omega}_{\alpha ,0} (e_{\alpha})=\bar{\omega}_{\beta ,0}(e_{\beta}) \quad \mathrm{and} \quad \bar{\omega}_{\alpha ,1} (e_{\alpha})=\bar{\omega}_{\beta ,1}(e_{\beta})
$$ 
or 
$$
\bar{\omega}_{\alpha ,0} (e_{\alpha})=\bar{\omega}_{\beta ,1}(e_{\beta}) \quad \mathrm{and} \quad  \bar{\omega}_{\alpha ,1} (e_{\alpha})=\bar{\omega}_{\beta ,0}(e_{\beta}). 
$$
Hence replacing $\theta$ with $\theta \circ \hat{\alpha}$ if necessary, we have 
$\bar{\omega} (\theta (e_{\alpha}))= \bar{\omega} (e_{\beta})$ for any $\omega\in T_1(\mathcal{W}_2\rtimes_{\beta}\mathbb{Z}_2)$. 
Therefore $\alpha$ is conjugate to $\beta$ by Proposition \ref{pro:comparison multiplier} and  Corollary \ref{cor:app-classification}. 
\end{proof}
\begin{rem}
(i) With notation as above, $\mathcal{W}_2\rtimes_{\alpha}\mathbb{Z}_2$ is isomorphic to $\mathcal{W}_2\rtimes_{\beta}\mathbb{Z}_2$ if and only if 
$T(\mathcal{W}_2\rtimes_{\alpha}\mathbb{Z}_2)\cong T(\mathcal{W}_2\rtimes_{\beta}\mathbb{Z}_2)$ as topological cones and 
$K_0(\mathcal{W}_2\rtimes_{\alpha}\mathbb{Z}_2)\cong K_0(\mathcal{W}_2\rtimes_{\beta}\mathbb{Z}_2)$ as abelian groups, and that these isomorphisms are compatible with the pairing 
between $T$ and $K_0$ by Lemma \ref{lem:closed-class} and Robert's classification theorem. \\ 
(ii) If $\mathcal{W}_2\rtimes_{\alpha}\mathbb{Z}_2$ and $\mathcal{W}_2\rtimes_{\beta}\mathbb{Z}_2$ have a unique tracial state, 
then $\alpha$ is conjugate to $\beta$ if and only if $\mathcal{W}_2\rtimes_{\alpha}\mathbb{Z}_2$ is isomorphic to $\mathcal{W}_2\rtimes_{\beta}\mathbb{Z}_2$. 
\end{rem}

We shall consider the classification up to cocycle conjugacy. 

\begin{pro}\label{pro:cocycle-dual}
Let $A$ be a simple exact $\sigma$-unital stably projectionless C$^*$-algebra with a unique tracial state and no unbounded traces, and let 
$\alpha$ and $\beta$ be actions of a finite abelian group $G$ on $A$. 
Assume that $A$ has strict comparison and almost stable rank one. 
Then $\alpha$ is cocycle conjugate to $\beta$ if and only if $\hat{\alpha}$ is conjugate to $\hat{\beta}$. 
\end{pro}
\begin{proof}
In general, it is easy to see that if $\alpha$ is cocycle conjugate to $\beta$, then $\hat{\alpha}$ is conjugate to $\hat{\beta}$. 
We will show the if part. Since we see that $\hat{\hat{\alpha}}$ is conjugate to $\hat{\hat{\beta}}$, there exists an automorphism $\theta$ of $A\otimes K(\ell^2(G))$ 
such that $\theta$ intertwines $\alpha\otimes \mathrm{Ad}(\rho )$ and $\beta\otimes \mathrm{Ad}(\rho )$ by the Takesaki-Takai duality theorem. 
Note that $A\otimes K(\ell^2(G))$ is a simple exact $\sigma$-unital stably projectionless C$^*$-algebra with a unique tracial state $\omega$ and no unbounded traces, which has 
strict comparison and almost stable rank one. Let $e\in K(\ell^2(G))$ be the projection onto the constant functions. 
Then we have 
$$
(A,\alpha) \cong (\bar{\theta}(1\otimes e)(A\otimes K(\ell^2(G))) \bar{\theta}(1\otimes e),  \beta\otimes \mathrm{Ad}(\rho )|_{\bar{\theta}(1\otimes e)(A\otimes K(\ell^2(G))) \bar{\theta}(1\otimes e)} )
$$
by the proof of Corollary \ref{cor:app-classification}. 
Since $\omega$ is the unique tracial state on $A\otimes K(\ell^2(G))$, we have $\bar{\omega}(\bar{\theta} (1\otimes e)) = \bar{\omega}(1\otimes e)$. 
By Proposition \ref{pro:comparison multiplier}, there exists a partial isometry $v$ in $M(A\otimes K(\ell^2(G)))$ such that 
$$
v^*v=\bar{\theta} (1\otimes e)\quad \mathrm{and}\quad vv^*=1\otimes e.
$$ For any $g\in G$, define 
$$
u (g):=v\overline{\beta_g\otimes\mathrm{Ad}\rho_g}(v^*).
$$
Then we have $(1\otimes e)u(g) (1\otimes e)=u (g)$ for any $g\in G$, and we may regard $u$ as a $\beta$-cocycle in $M(A)$. 
It can be easily checked that 
$$
(\bar{\theta}(1\otimes e)(A\otimes K(\ell^2(G))) \bar{\theta}(1\otimes e),  \beta\otimes \mathrm{Ad}(\rho )|_{\bar{\theta}(1\otimes e)(A\otimes K(\ell^2(G))) \bar{\theta}(1\otimes e)} ) 
\cong (A, \beta^{u}),
$$
and hence we obtain the conclusion. 
\end{proof}

\begin{thm}\label{thm:cocycle-conjugacy}
Let $\alpha$ and $\beta$ be outer actions of $\mathbb{Z}_2$ on $\mathcal{W}_2$. Assume that $\alpha$ and $\beta$ are locally representable in $\mathcal{C}_{\mathrm{R}}$. 
Then $\alpha$ is cocycle conjugate to $\beta$ if and only if $\mathcal{W}_2\rtimes_{\alpha}\mathbb{Z}_2$ is isomorphic to $\mathcal{W}_2\rtimes_{\beta}\mathbb{Z}_2$. 
\end{thm}
\begin{proof}
The only if part is obvious. We will show the if part. Let $\theta$ be an isomorphism from $\mathcal{W}_2\rtimes_{\alpha}\mathbb{Z}_2$ onto $\mathcal{W}_2\rtimes_{\beta}\mathbb{Z}_2$. 
The proof of Theorem \ref{thm:Main-sym} implies that $\hat{\beta}$ is approximately unitarily equivalent to $\theta\circ \hat{\alpha} \circ \theta^{-1}$. 
Hence $\hat{\alpha}$ is conjugate to $\hat{\beta}$ by Theorem \ref{thm:Rohlin} and Proposition \ref{thm:dual}. 
Proposition \ref{pro:cocycle-dual} shows that $\alpha$ is cocycle conjugate to $\beta$. 
\end{proof}

The following example is based on Blackadar's construction of symmetries on $M_{2^{\infty}}$ in \cite{Bla2} and Robert's classification theorem. 

\begin{ex}\label{ex:dyadic-rational}
\cite[Theorem 5.2.1 and Theorem 5.2.2]{Ell} show that there exists a simple stably projectionless C$^*$-algebra $A$ such that $A$ is expressible as an inductive limit C$^*$-algebra of 
1-dimensional NCCW complexes with trivial $K_1$-groups and $(K_0(A),(T(A),T_1(A)),r_A)= (\mathbb{Z},(\mathbb{R}_{+},\emptyset),0)$. 
By \cite[Proposition 5.2]{Na2}, there exists a hereditary subalgebra $B$ of $A$ such that $B$ has a unique tracial state $\omega$ and no unbounded traces. 
Since $B$ is stably isomorphic to $A$ by Brown's theorem \cite{B}, $B$ is contained in $\mathcal{C}_{\mathrm{R}}$ and $(K_0(B),(T(B),T_1(B)),r_B)= (\mathbb{Z},(\mathbb{R}_{+},\{\omega\}),0)$. 

By Robert's classification theorem (Theorem \ref{thm:Robert}), there exists an automorphism $\beta$ of $B$ such that $K_0(\beta)=-\mathrm{id}$. 
Define a homomorphism $\Phi_i$ from $M_{2^i}(\mathbb{C})\otimes B$ to $M_{2^{i+1}}(\mathbb{C})\otimes B$ by 
$$
\Phi_i (x) =  \left(\begin{array}{cc}
            x    &      0          \\ 
            0    &     (\mathrm{id}_{M_{2^{i}}(\mathbb{C})}\otimes\beta ) (x)   
 \end{array} \right) ,
$$
and put $D:=\displaystyle{\lim_{\rightarrow}}(M_{2^i}(\mathbb{C} )\otimes B, \Phi_i)$. 
Then we have $K_0(\Phi_i )= 0$, and hence $K_0(D)=0$. 
It can be easily checked that $D$ is a simple stably projectionless C$^*$-algebra with a unique tracial state and no unbounded traces. 
Therefore $D$ is isomorphic to $\mathcal{W}_2$ by Robert's classification theorem. 

We shall construct a locally representable $\mathbb{Z}_2$-action $(\mathcal{W}_2, \alpha )$ in $\mathcal{C}_{\mathrm{R}}$. 
It is easy to see that $\Phi_i$ is unitarily equivalent to a homomorphism $\Psi_i$ such that 
$$
\Psi_i ( 
            \left(\begin{array}{cc}
            x_{11}    &     x_{12}         \\ 
            x_{21}    &     x_{22}   
 \end{array} \right) )
= 
 \left(\begin{array}{cccc}
                  x_{11}    &             0                &  x_{12}   &             0                     \\ 
                  0            &     (\mathrm{id}_{M_{2^{i-1}}(\mathbb{C})}\otimes\beta )  (x_{22}) &        0           &  (\mathrm{id}_{M_{2^{i-1}}(\mathbb{C})}\otimes\beta ) (x_{21})       \\
                  x_{21}    &             0                &  x_{22}   &             0                     \\
                  0            & (\mathrm{id}_{M_{2^{i-1}}(\mathbb{C})}\otimes\beta )  (x_{12}) &        0           &  (\mathrm{id}_{M_{2^{i-1}}(\mathbb{C})}\otimes\beta ) (x_{11})
 \end{array} \right) 
$$
for any $ \left(\begin{array}{cc}
            x_{11}    &     x_{12}         \\ 
            x_{21}    &     x_{22}   
 \end{array} \right) \in M_{2^i}(\mathbb{C})\otimes B$. 
Since $\displaystyle{\lim_{\rightarrow}}(M_{2^i}(\mathbb{C} )\otimes B, \Psi_i)$ is isomorphic to $D$, 
$\displaystyle{\lim_{\rightarrow}}(M_{2^i}(\mathbb{C} )\otimes B, \Psi_i)$ is isomorphic to $\mathcal{W}_2$. 
For any $i\in\mathbb{N}$, define an automorphism $\alpha_i$ of $M_{2^i}(\mathbb{C})\otimes B$ by 
$$
\alpha_i ( \left(\begin{array}{cc}
            x_{11}    &     x_{12}         \\ 
            x_{21}    &     x_{22}   
 \end{array} \right) ) 
:= 
\mathrm{Ad} ( \left(\begin{array}{cc}
            1    &      0         \\ 
            0    &     -1   
 \end{array} \right) ) 
( \left(\begin{array}{cc}
            x_{11}    &     x_{12}         \\ 
            x_{21}    &     x_{22}   
 \end{array} \right)
)
$$
for any $ \left(\begin{array}{cc}
            x_{11}    &     x_{12}         \\ 
            x_{21}    &     x_{22}   
 \end{array} \right) \in M_{2^i}(\mathbb{C})\otimes B$. 
Then we have $\Psi_i \circ \alpha_i = \alpha_{i+1} \circ \Psi_i$, and hence there exists an automorphism $\alpha$ of $\mathcal{W}_2$ such that 
$\alpha (a) = \alpha_{i} (a)$ for any $a\in M_{2^i}(\mathbb{C})\otimes B$ and $i\in\mathbb{N}$.  
It is obvious that $\alpha$ is a locally representable $\mathbb{Z}_2$-action in $\mathcal{C}_{\mathrm{R}}$. 

We shall consider the fixed point algebra $\mathcal{W}_2^{\alpha}$. It can be easily checked that $\mathcal{W}_2^{\alpha}$ is isomorphic to an inductive limit C$^*$-algebra 
$\displaystyle{\lim_{\rightarrow}}(M_{2^{i-1}}(\mathbb{C})\otimes B\oplus M_{2^{i-1}}(\mathbb{C})\otimes B, \Theta_i)$ where 
$$ 
\Theta_i ((x_1,x_2)) = ( \left(\begin{array}{cc}
                  x_1    &     0         \\ 
                  0       &  (\mathrm{id}_{M_{2^{i-1}}(\mathbb{C})}\otimes\beta ) (x_2)      
 \end{array} \right) ,
\left(\begin{array}{cc}
                  x_2    &     0         \\ 
                  0       &  (\mathrm{id}_{M_{2^{i-1}}(\mathbb{C})}\otimes\beta ) (x_1)      
 \end{array} \right) )
$$
for any $(x_1,x_2) \in M_{2^{i-1}}(\mathbb{C})\otimes B\oplus M_{2^{i-1}}(\mathbb{C})\otimes B$. 
Therefore we have $K_0(\mathcal{W}_2^\alpha )= \mathbb{Z} [\frac{1}{2}]$. 
It is easy to see that $\alpha$ is outer, and hence $\mathcal{W}_2\rtimes_{\alpha}\mathbb{Z}_2$ is stably isomorphic to $\mathcal{W}_2^\alpha$. 
Consequently we obtain a locally representable outer $\mathbb{Z}_2$-action 
$(\mathcal{W}_2, \alpha )$ in $\mathcal{C}_{\mathrm{R}}$ such that 
$$
K_0(\mathcal{W}_2\rtimes_{\alpha} \mathbb{Z}_2) =\mathbb{Z}[\frac{1}{2}] \quad \mathrm{and} \quad K_1(\mathcal{W}_2\rtimes_{\alpha} \mathbb{Z}_2) =0.
$$
Note that the dual action $\hat{\alpha}$ on $\mathcal{W}_2\rtimes_{\alpha}\mathbb{Z}_2$ has the Rohlin property by Proposition \ref{thm:dual}. 
\end{ex}

\begin{rem}
Blackadar showed that if $A$ is a C$^*$-algebra with $K_0(A)=K_1(A)=0$ and $\beta$ is a $\mathbb{Z}_2$-action on $A$, 
then $K_0(A\rtimes_{\beta}\mathbb{Z}_2)$ and $K_1(A\rtimes_{\beta}\mathbb{Z}_2)$ are uniquely 2-divisible (see \cite[Lemma 4.4]{I1}). (Note that 
an abelian group $\Gamma$ is uniquely 2-divisible if and only if $\Gamma\otimes_{\mathbb{Z}} \mathbb{Z}[\frac{1}{2}]$ is isomorphic to $\Gamma$.) 
For every pair of countable abelian uniquely 2-divisible groups $\Gamma_0$ and $\Gamma_1$,  Izumi constructed an approximately representable outer action $\beta$ of $\mathbb{Z}_2$ on 
$\mathcal{O}_2$ such that $K_0(\mathcal{O}_2\rtimes_{\beta}\mathbb{Z}_2)=\Gamma_0$ and $K_1(\mathcal{O}_2\rtimes_{\beta}\mathbb{Z}_2)=\Gamma_1$ 
\cite[Theorem 4.8 (3)]{I1}. If the Robert conjecture is true, then we can construct an approximately representable outer action $\beta$ of $\mathbb{Z}_2$ on 
$\mathcal{W}_2$ such that $K_0(\mathcal{W}_2\rtimes_{\beta}\mathbb{Z}_2)=\Gamma_0$ and $K_1(\mathcal{W}_2\rtimes_{\beta}\mathbb{Z}_2)=\Gamma_1$ by using 
$\alpha$ in Example \ref{ex:dyadic-rational} and the same construction of Izumi. Indeed, there exists a simple separable nuclear stably projectionless C$^*$-algebra $B$ with 
a unique tracial state and no unbounded traces such that $K_0(B)=\Gamma_0$ and $K_1(B)=\Gamma_1$ (see, for example, \cite{K2} and \cite[Proposition 5.2]{Na2}). 
If the Robert conjecture is true, then $B\otimes \mathcal{W}_2$ is isomorphic to $\mathcal{W}_2$. Define an action $\beta$ on $B\otimes\mathcal{W}_2\cong \mathcal{W}_2$ by 
$\beta :=\mathrm{id}\otimes \alpha$. Then $\beta$ is an approximately representable outer action such that 
$K_0(\mathcal{W}_2\rtimes_{\beta}\mathbb{Z}_2)=\Gamma_0$ and $K_1(\mathcal{W}_2\rtimes_{\beta}\mathbb{Z}_2)=\Gamma_1$ by the Kunneth formula. 
\end{rem}

\section*{Acknowledgments}
The author would like to thank the people in University of Copenhagen, where this work was done, 
for their hospitality. He is also grateful to Hiroki Matui for many helpful discussions and valuable suggestions.

\end{document}